\documentclass{amsart}
\usepackage{color}

\newtheorem{thm}{Theorem}[section]

\newtheorem{cor}[thm]{Corollary}
\newtheorem{lem}[thm]{Lemma}
\newtheorem{prop}[thm]{Proposition}
\newtheorem{defn}[thm]{Definition}
\newtheorem{rem}[thm]{Remark}

\newcommand{\norm}[1]{\left\Vert#1\right\Vert}
\newcommand{\abs}[1]{\left\vert#1\right\vert}
\newcommand{\set}[1]{\left\{#1\right\}}

\newcommand{\pfrac}[2]{\frac{\partial #1}{\partial #2}}

\begin{document}                                                                                   
\title{Ricci flow on surfaces with conical singularities, II}

\numberwithin{equation}{section}
\author{Hao Yin}

\address{Hao Yin, 
School of Mathematical Science, University of Science and Technology of China, Hefei, China}
\email{haoyin@ustc.edu.cn}

\begin{abstract}
  In this paper, we study the (normalized) Ricci flow on surfaces with conical singularities. Long time existence is proved for cone angle smaller than $2\pi$. In this case, convergence results are obtained if the Euler number is nonpositive. 
\end{abstract}
\subjclass{AMS 53C44} \keywords{Ricci flow,
Conical singularity}
\thanks{Research supported by NSFC 11101272.}

 \maketitle

 \section{Introduction}\label{sec:intro}

 This is a continuation of our previous work on Ricci flow on surfaces with conical singularities \cite{mine}. For smooth surfaces, in 1982, Hamilton \cite{Ham} used his Ricci flow to prove the (normalized) Ricci flow converges to a metric of constant curvature in a given conformal class. Chow \cite{Ben}, Chen, Lu and Tian \cite{CLT} removed various technical conditions to show that Ricci flow can be used as a tool to prove the Uniformization theorem. 

 Let $S$ be a smooth Riemann surface and $p\in S$. A metric $g$ on $S$ is said to have a conical singularity of order $\beta(\beta>-1)$, or of angle $\theta=(2\pi(\beta+1))$ at $p$, if in a neighborhood $U$,
 \begin{equation*}
   g=w(z)\abs{z}^{2\beta}\abs{dz}^2,
 \end{equation*}
 where $z$ is a conformal coordinate in $U$ with $z(p)=0$ and $w$ is some continuous positive function. In general, consider a divisor $\beta=\sum_{i=1}^n \beta_i p_i$ and a conformal metric $g$ on $S$ which has a conical singularity of order $\beta_i$ at $p_i$. The set of all such metrics is a singular conformal class, which we denote by $(S,\beta)$. 
 
 We are interested in generalizing the theory about Ricci flow to surfaces with conical singularities. In particular, we hope Ricci flow could be a useful tool in the study of canonical metric in the conformal class $(S,\beta)$. There are already several papers devoted to this subject, for example, \cite{LuoTian}, \cite{threePoint}, \cite{Troyanov}. It suffices for us to note here that not all $(S,\beta)$ admits a metric of constant curvature and it is  interesting to find a `canonical' metric in $(S,\beta)$ for general $S$ and $\beta$. { After a first version of this paper was submitted for publication, there are substential progress made in this direction. The existence of canonical metrics is related to some stability notion. We refer to \cite{MRS} and \cite{Song} in this direction.} 

 { The concial metric that we discuss is not complete. Geometric flows starting from incomplete metrics are discussed by many other authors, for example, M. Simon \cite{MS1,MS2}, E.Bahuaud and B. Vertman \cite{BV}, Topping \cite{To} and G. Giesen and P. Topping \cite{GT1}. In \cite{MS1,MS2}, a conical singularity is smoothed out by the Ricci flow to get a smooth solution for $t>0$. In \cite{To} and \cite{GT1}, the metric becomes complete noncompact instantaneously. Hence, the singularity is gone as long as $t>0$ in the above mentioned papers. However, in this paper, as in \cite{mine} and \cite{BV}, the singularity remains there along the flow.

 To preserve the singularity along the flow, a special set of analysis tools is needed. Such analysis features a uniform degeneration of ellipticity near the singularity. Rather complete theory for such linear operators has been developed for various purposes, for example, Melrose \cite{Melrose}, Mazzeo \cite{Mazzeo}, Cheeger \cite{Cheeger} and Schulze \cite{Sch}. The microlocal analysis method from \cite{Melrose, Mazzeo} is generalized to parabolic equations in \cite{Moo, BDB}. Recently, motivated by his program in K\"ahler geometry, Donaldson \cite{Donaldson} proved a Schauder estimate in this setting. This approach is taken by Chen and Wang \cite{new,CW2} to study K\"ahler Ricci flow on K\"ahler manifold singular along a smooth divisor. Besides the analysis theory mentioned above, this paper together with its part one \cite{mine} provide a set of results and techniques which supports the study of semi-linear parabolic equation on surfaces with conical singularities. In contrast with the above mentioned analysis theory, the method we use here is elementary and completely self-contained. Hopefully, it is of some interests for readers with less analysis background.
}

In \cite{mine}, the author showed that there is a reasonable notion of normalized Ricci flow despite that the metric is not complete at the cone tips. For a `good' initial value, there is a family of metrics $g(t), t\in [0,T]$ in the conformal class $(S,\beta)$ satisfying on $S\setminus \set{p_i}$
 \begin{equation}\label{eqn:nrf}
   \pfrac{g}{t}=(r-R)g.
 \end{equation}
 Here $R$ is the scalar curvature of $g(t)$ and $r$ is some constant. Moreover, the Gauss-Bonnet formula remains true for $g(t)$ and the volume of $g(t)$ does not depend on $t$. Precisely, we have
 \begin{equation*}
	 \int_{S\setminus \set{p_i}} K_g dV_g =2\pi \tilde{\chi} =\frac{r}{2} V(g_t),
 \end{equation*}
 where $K_g$ is the Gauss curvature, $V(g_t)$ is the total area of $g_t$ and $\tilde{\chi}=\chi(S)+\sum_{i=1}^n \beta_i$ is the Euler number of $(S,\beta)$. 

 Due to technical difficulties, in \cite{mine}, the author was not able to show any long time existence result. It was not even clear whether the curvature is bounded along the flow. The main purpose of this paper is to solve these problems. In fact, we reprove the local existence result by iterating in a smaller function space. The solution obtained this way automatically has bounded Gauss curvature. We then discuss long time existence and convergence result. We follow closely the argument in \cite{Ham}. With the presence of conical singularity, we must be careful in the application of maximum principle. It turns out that for cone singularities with $\beta_i<0$, some arguments in \cite{Ham} work. 


 To introduce our main result, we set up some notions. Throughout this paper, we choose a background cone metric $\tilde{g}$ in the conformal class $(S,\beta)$ which is smooth away from the singularity and is the standard flat cone metric in small neighborhoods of the singularities. That is, if $p$ is a singular point, then there exists some complex coordinate $z$ compatible to the complex structure of $S$ near $p$, then
 \begin{equation*}
   \tilde{g}=\abs{z}^{2\beta} \abs{dz}^2.
 \end{equation*}
 We denote the Laplace of $\tilde{g}$ by $\tilde{\triangle}$, the Gauss curvature by $\tilde{K}$, the gradient by $\tilde{\nabla}$ and the volume form by $d\tilde{V}$. The initial metric of the normalized Ricci flow is denoted by $g_0=e^{2u_0}\tilde{g}$  and the solution by $g(x,t)$. $K_t$ and $dV_t$ are the Gauss curvature and volume element of $g_t$ for $t\geq 0$.

 Our result is
 \begin{thm}
   \label{thm:main}
   Assume both $u_0$ and $K_0$ are in $C^{0,\alpha} (S,\beta)$ and $u_0, K_0$ and $\triangle K_0$ are bounded with bounded Dirichlet energy. Then there exists a local solution to (\ref{eqn:nrf}) defined on $[0,T]$. If we choose $r$ properly, the volume of $g(t)$ will be a constant and Gauss-Bonnet theorem remains true for $t>0$. Moreover,
   
   (1) if $K_0<c<0$, then the solution exists for all $t$ and converges to a metric of constant negative curvature.

   (2) if $\beta_i<0$ for all $i$, then the solution exists for all $t$.

   (3) if $\beta_i<0$ for all $i$ and $\chi (S,\beta)\leq 0$, then the solution converges to a metric of constant curvature.
 \end{thm}

 The proof of this theorem depends on further understanding of both the analysis involved in the cone singularity and the nature of the Ricci flow equation (on surface). The central theme is that what regularity can be proved for the singular equation  and whether this regularity enables us to reproduce the methods used in the smooth setting. The idea behind our previous paper on this subject \cite{mine} can be summarized as follows: (1) approximate the cone manifold by manifold with boundary; (2) obtain uniform $C^0$ estimate in the approximation and leave higher order estimates to interior estimates away from the singularity; (3) show the finiteness of Dirichlet energy, which justifies the maximum principle.

 In this paper, several new observations and methods are introduced. For the convenience of the reader, we list them below.

 (1) The approximation scheme we used in \cite{mine} is robust if we take $t$-derivative of the evolution equation. Precisely, we approximate $u$ (the solution on conical surface) by $u_k$ (the solution on the manifold with boundary), while at the same time $\partial_t u$ is approximated by $\partial_t u_k$. This enables us to prove linear estimate for $\partial_t u$ and the Dirichlet energy of $\partial_t u$ (see Lemma \ref{lem:two} and Lemma \ref{lem:energytwo}) in Section \ref{sec:linear}. Using this double-layer linear estimate, we do iteration to obtain local solution to the normalized Ricci flow with bounded curvature in Section \ref{sec:iteration}.

 (2) In principle, we can do multi-layer linear estimate by keep taking $t$-derivatives. In that sense, we can obtain local solution with better and better regularity, at the expenses of assuming more and more of the initial data. A difficulty is that to obtain a local solution with bounded curvature $K$, we need to assume that $\triangle_0 K_0$ is bounded for the initial value (see Theorem \ref{thm:short}). Therefore, even if we use the evolution equation of $K$ to show that $K$ is bounded up to time $T$, Theorem \ref{thm:short} does not extend the solution beyond $T$. 

 This problem can be solved as follows. We note that the normalized Ricci flow equation is quasi-linear, the evolution equation for $\partial_t u$ (or equivalently, $\triangle u$, or $K$) is semi-linear and fortunately, the evolution equation for $\partial_t K$ (or equivalently $\triangle K$ if $u$ and $K$ are bounded) is linear (with $u$ and $K$ as coefficients). The point is that linear equations never blow up. Our proof of Lemma \ref{lem:long} follows this line. 

 (3) This last observation is not a necessary logical part in the proof. It is a point of view about the regularity of parabolic equations on conical surfaces. If $u$ is a solution to (\ref{eqn:nrf}), what is the regularity near the cone point? Is the gradient of $u$ bounded? the Hessian? A `correct' way of answering these questions is that $t$-derivatives of $u$ are bounded (as long as $u$ is) and one should consult the regularity of Poisson equation for the boundedness of $\nabla u$ and $\nabla^2 u$, which we know depends on the cone angle. We discuss the Poisson equation in Section \ref{sec:poisson}. We need it to justify the application of maximum principle to $\abs{\nabla f}^2$, where $f$ is the potential of curvature ($\triangle f= r-R$).

 Other proofs in this paper are more or less known techniques in the theory of Ricci flow on surfaces.

 It is natural to ask what happens for positive Euler number case. With the existence of conical singularities, there may not be a constant curvature metric. It is expected that when constant curvature metric exists, the flow converges to it and when no such metric exists, the flow approximates a Ricci soliton. We refer the readers to the recent preprints \cite{MRS} and \cite{Song} for the latest progress made in this direction.  

 \begin{rem}
	 After the completion of this paper, we found a gap in our previous paper \cite{mine}. In the proof of local existence, we used Schauder fixed point theorem. However, we did not check that the map constructed is a continuous map. Moreover, we do not know how to fix the problem in that context. The results of this paper do not depend on that theorem and provide a fix to the problem. It turns out that the stronger space of iteration is the key to this problem.
 \end{rem}

 \section{Linear parabolic equation and maximum principle}\label{sec:linear}

In this section, we study the following linear parabolic equation
\begin{equation}\label{eqn:one}
  \pfrac{u}{t}= a(x,t)\tilde{\triangle} u+ b(x,t)u +f(x,t)
\end{equation}
with initial value
\begin{equation*}
	u|_{t=0}=u_0
\end{equation*}
on $S$,
where $\tilde{\triangle}$ is the Laplacian of the fixed background metric $\tilde{g}$.

\subsection{Basic estimates}

Important to our discussion are some weighted H\"older spaces defined in \cite{mine}. The elliptic version is $C^{l,\alpha}(S,\beta)$ and the parabolic version is $C^{l,\alpha}( (S,\beta)\times [0,T])$. In this paper, for simplicity, we denote them by $\mathcal E^{l,\alpha}$ and $\mathcal P^{l,\alpha,T}$ respectively. 
{For reader's convenience, we recall the definition .

For simplicity, we assume there is only one singular point $p$ and $U$ is a neighborhood of $p$, in which $(x,y)$ is a local coordiante system compatible with the conformal structure of $S$ and $p=(0,0)$. By setting
\begin{equation*}
	x=r\cos \theta,\quad y=r\sin \theta
\end{equation*}
and
\begin{equation*}
	\rho=\frac{1}{\beta+1}r^{\beta+1},
\end{equation*}
we regard $f$ as a function of $(\rho,\theta)$. Assume that $(\rho,\theta)$ is defined for $\rho\leq 2$. If we write $B_2\setminus B_1$ for the annulus defined by $1\leq \rho\leq 2$,
then the $\mathcal E^{l,\alpha}$ norm is defined as
\begin{equation*}
	\norm{f}_{\mathcal E^{l,\alpha}}= \sup_{k=0 \ldots \infty} \norm{f(2^{-k} \rho,\theta)}_{C^{l,\alpha}(B_2\setminus B_1)} + \norm{f}_{C^{l,\alpha}(S\setminus U)}.
\end{equation*}
Here $C^{l,\alpha}(B_2\setminus B_1)$ and $C^{l,\alpha}(S\setminus U)$ are just usual H\"older norms. Similarly, if $f$ is a function defined on space time $S\times [0,T]$, we define (by regarding $f$ as a function of $(\rho,\theta,t)$)
\begin{equation*}
	\norm{f}_{\mathcal P^{l,\alpha}}= \sup_{k=0 \ldots \infty} \norm{f(2^{-k} \rho,\theta,4^{-k} t)}_{C^{l,\alpha}(B_2\setminus B_1\times [0,4^k T])} + \norm{f}_{C^{l,\alpha}(S\setminus U\times [0,T])}.
\end{equation*}
To see that the definition is independent of our choice of coordinate system $(x,y)$, we refer to Section 2 of \cite{mine}.
}

In spite of the tedious definition, it is not difficult to understand the meaning of these weighted H\"older space. Away from the singularity, they are just the normal H\"older space. Near a singularity, the $\mathcal E^{l,\alpha}$ norm is the bound for up to $l-$th derivatives which one may obtain for a bounded harmonic function via applying interior estimate on a ball away from the singularity. A similar characterization is true for $\mathcal P^{l,\alpha,T}$ if we replace the harmonic function by a solution to linear heat equation defined on $S\times [0,T]$. 

\begin{rem}
	Weighted H\"older space is nothing new in the study of degenerate elliptic operators. In fact, $\mathcal E^{l,\alpha}$ coincides with the edge H\"older space in \cite{big3} in the case of conical surfaces. See also Section 3.2 and Section 3.6.2 of \cite{MRS}.
\end{rem}

The following is almost Theorem 3.1 in \cite{mine}. The difference is that we have an extra $b(x,t)u$ term in our equation, which does not affect the proof so that the proof in \cite{mine} works here with almost no modification. {We refer to Section 5 of \cite{JL} and Section 3 of \cite{MRS} for similar Schauder type estimates.}
\begin{lem}\label{lem:one}
  If $a,b,f\in \mathcal P^{0,\alpha,T}$ and $u_0\in \mathcal E^{2,\alpha}$, then there exists a solution $u$ to equation (\ref{eqn:one}) with initial value $u_0$ such that
  \begin{equation*}
    \norm{u}_{\mathcal P^{2,\alpha,T}}\leq C
  \end{equation*}
  where $C$ depends on $T$ and the $\mathcal P^{0,\alpha,T}$ norm of $a,b,f$ and the $\mathcal E^{2,\alpha}$ norm of $u_0$.
\end{lem}

\begin{rem}
  In fact, $C$ depends on $\max(T,1)$ instead of $T$. It follows from the $C^0$ estimate in the proof. We shall need this fact later. 
\end{rem}

To illustrate the idea of the proof, assume that there is only one singularity. On a neighborhood $U$ of the singularity, the background metric $\tilde{g}$ is given by
 \begin{equation*}
   \tilde{g}=r^{2\beta}(dr^2+r^2 d\theta^2).
 \end{equation*}
Set
\begin{equation*}
  S_k=S\setminus \{(x,y)\in U |\rho(x,y)<2^{-k}\}
\end{equation*}
where $\rho=\frac{1}{\beta+1}r^{\beta+1}$.
We study the following boundary value problem
\begin{equation}\label{eqn:sk}
  \left\{
  \begin{array}[]{ll}
    \partial_t u =a(x,t)\tilde{\triangle} u +b(x,t)u + f(x,t) & \mbox{on } S_k \\
    \pfrac{u}{\nu} |_{\partial S_k}=0 & \\
    u(x,0)=u_0(x)& \mbox{ on } S_k
  \end{array}
  \right.
\end{equation}
Here $\nu$ is the normal vector to the boundary $\partial S_k$. 
It is proved in the Appendix (Theorem \ref{app:good}) that there is a solution $u_k$ to  (\ref{eqn:sk}). $u_k$ is $C^{2,\alpha}$ away from $\partial S_k\times \set{0}$.  
We can obtain (see Remark \ref{rem:maximum}) a uniform $C^0$ estimate for $u_k$ so that for each compact set $K\subset S\setminus \set{p_i}$ and $k$ sufficiently large, $C^{2,\alpha}(K\times [0,T])$ estimate follows from interior estimates. By sending $k$ to infinity, we get the solution $u$ in Lemma \ref{lem:one}. The estimate of $u$ follows from the interior estimates again.

The next lemma is an estimate for $\int_{S\setminus \set{p_i}} \abs{\tilde{\nabla} u}^2 d\tilde{V}$. It turns out that the boundedness of the Dirichlet energy is very important to us. For example, we need some maximum principle for parabolic equations. However, the presence of a conical singularity implies that ordinary maximum principle can't be true without extra assumptions. Indeed, being bounded in the Dirichlet energy is part of the assumption. For more detail, see the next subsection on maximum principle.

\begin{lem}
\label{lem:energy}
Let $u$ be the solution constructed in Lemma \ref{lem:one} and assume $u_0$, $b$ and $f(\cdot,t)$ have uniformly bounded Dirichlet energy. 
Then 
{\begin{eqnarray}\label{eqn:energy}
	&&\int_{S\setminus \set{p_i}} \abs{\tilde{\nabla} u}^2(t) d\tilde{V}\\ \nonumber
  &\leq& e^{C_1t} \int_{S\setminus \set{p_i}} \abs{\tilde{\nabla} u_0}^2 d\tilde{V} +\int_0^t e^{C_1(t-s)} \left( C_2 \int_{S\setminus \set{p_i}} \abs{\tilde{\nabla} f}^2 +\abs{\tilde{\nabla} b}^2 d\tilde{V}\right) ds. 
\end{eqnarray}
}
Here $C_1$ depends on the $C^0$ norm of $b$ and $C_2$ depends on the $C^0$ norm of $u$.
In particular, $u(\cdot,t)$ has bounded energy.
\end{lem}

\begin{proof}
	Recall that $u$ is the limit {of} $u_k$. Moreover,
\begin{eqnarray*}
  \frac{d}{dt} \int_{S_k} \abs{\tilde{\nabla} u_k}^2 d\tilde{V}&=& 2\int_{S_k} \tilde{\nabla} u_k \cdot \tilde{\nabla} \pfrac{u_k}{t} d\tilde{V} \\
  &= & -2\int_{S_k} a(x,t)(\tilde{\triangle} u_k)^2 d\tilde{V} + C \int_{S_k} \tilde{\nabla} u_k\cdot \tilde{\nabla} (b u_k) d\tilde{V} +2 \int_{S_k} \tilde{\nabla} u_k \cdot \tilde{\nabla} f(x,t) d\tilde{V} \\
  &\leq& C_1 \int_{S_k}\abs{\tilde{\nabla} u_k}^2 d\tilde{V} +C_2 \int_{S_k} \abs{\tilde{\nabla} f}^2 +\abs{\tilde{\nabla} b}^2 d\tilde{V}. 
\end{eqnarray*}
Here $C_1$ depends on the $C^0$ norm of $b$ and $C_2$ depends on the $C^0$ norm of $u_k$.

Integrating from $s=0$ to $s=t$, we get
\begin{eqnarray*}
  &&\int_{S_k} \abs{\tilde{\nabla} u_k}^2(t) d\tilde{V}\\ 
  &\leq& e^{C_1t} \int_{S_k} \abs{\tilde{\nabla} u_0}^2 d\tilde{V} +\int_0^t e^{C_1(t-s)} \left( C_2 \int_{S_k} \abs{\tilde{\nabla} f}^2 +\abs{\tilde{\nabla} b}^2 d\tilde{V}\right) ds. 
\end{eqnarray*}
\begin{rem}
	When integrating from $s=0$ to $s=t$ above, we implicitly assumed that
	\begin{equation*}
		\lim_{s\to 0} \int_{S_k} \abs{\tilde{\nabla} u_k}^2 d\tilde{V} = \int_{S_k} \abs{\tilde{\nabla} u_0}^2 d\tilde{V}.
	\end{equation*}
	This is proved in the Appendix(Theorem \ref{app:good}).
\end{rem}

The lemma is proved by taking $k\to \infty$ and noticing that
\begin{equation*}
	\int_{S\setminus \set{p_i}} \abs{\tilde{\nabla} u}^2 (t) d\tilde{V} \leq \liminf_{k\to \infty} \int_{S_k} \abs{\tilde{\nabla} u_k}^2 (t) d\tilde{V}
\end{equation*}
with equality for $t=0$.

\end{proof}

In \cite{mine}, we applied Lemma \ref{lem:one} and Lemma \ref{lem:energy} to a linear equation to show the local existence of the normalized Ricci flow. The linear equation was of the form
\begin{equation}\label{eqn:smallone}
  \pfrac{u}{t}=a(x,t)\tilde{\triangle} u + f(x,t). 
\end{equation}
We add an extra term $b(x,t)u$ in this paper because in addition to (\ref{eqn:smallone}), we would also like to study the evolution equation of $\pfrac{u}{t}$, which is of the form (\ref{eqn:one}). The estimate of $\pfrac{u}{t}$ will enable us to prove the local existence of the normalized Ricci flow in a smaller space, which guarantees, among other things, the boundedness of the Gauss curvature.
 
To illustrate the idea of getting estimate of $\pfrac{u}{t}$, we take $t$-derivative of (\ref{eqn:sk}) with $b=0$ to see
\begin{equation*}
  \left\{
  \begin{array}[]{ll}
    \partial_t (\partial_t u_k) =a(x,t)\tilde{\triangle} (\partial_t u_k) + \frac{\partial_t a}{a} (\partial_t u_k - f)  + \partial_t f(x,t) & \mbox{on } S_k \\
    \pfrac{(\partial_t u_k)}{\nu} |_{\partial S_k}=0 & \\
    (\partial_t u_k)(x,0)=a(x,0)\tilde{\triangle} u_0(x)+ f(x,0)& \mbox{on } S_k.
  \end{array}
  \right.
\end{equation*}

The important observation is that $\partial_t u_k$ still satisfies the Neumann boundary condition. Naively, one may want to argue as in Proposition 3.2 of \cite{mine} to derive estimates of $\partial_t u_k$ and pass it to $\partial_t u$, which satisfies
\begin{equation*}
  \left\{
  \begin{array}[]{ll}
    \partial_t (\partial_t u) =a(x,t)\tilde{\triangle} (\partial_t u) + \frac{\partial_t a}{a} (\partial_t u - f)  + \partial_t f(x,t) & \mbox{on } S \\
    (\partial_t u)(x,0)=a(x,0)\tilde{\triangle} u_0(x)+ f(x,0)& \mbox{on } S.
  \end{array}
  \right.
\end{equation*}

However, the compatibility condition for initial-boundary value problem of parabolic equation becomes serious here. Namely, if
\begin{equation*}
\partial_\nu  u_0 \ne 0
\end{equation*}
on $\partial S_k$, then $u_k$ can not be $C^1$ at the corner $\partial S_k \times \set{0}$. In fact, it is shown in the appendix that it is not (in general). Therefore, our methods of proving $C^0$ and $W^{1,2}$ bound fail for $\pfrac{u_k}{t}$. 

To solve this problem, we consider the equation
\begin{equation*}
  \left\{
  \begin{array}[]{ll}
    \partial_t w =a(x,t)\tilde{\triangle} w + \frac{\partial_t a}{a} (w - f)  + \partial_t f(x,t) & \mbox{on } S_k \\
    \pfrac{w}{\nu} |_{\partial S_k}=0 & \\
    w(x,0)=a(x,0)\tilde{\triangle} u_0(x)+ f(x,0)& \mbox{on } S_k.
  \end{array}
  \right.
\end{equation*}
This is exactly the equation satisfied by $\pfrac{u_k}{t}$. Theorem \ref{app:good} provides another solution denoted by $w_k$. $w_k$ is more regular at $t=0$ as given by (3) and (4) of Theorem \ref{app:good}. This allows us to prove linear estimates (including $C^0$, $\mathcal P^{2,\alpha,T}$ and energy) for $w$. The key observation is
\begin{equation*}
	\lim_{k\to \infty} w_k(x,t)= \lim_{k\to \infty} \pfrac{u_k}{t} (x,t)= \pfrac{u}{t}(x,t).
\end{equation*}
This follows from
  \begin{equation}\label{eqn:appendix}
	  \int_{S_k\times [0,T]} \abs{w_k-\pfrac{u_k}{t}} dVdt \leq C \int_{\partial S_k} \abs{\nabla u_0} dV_{\partial S_k}.
  \end{equation}
  We will give a detailed proof of this in the Appendix \ref{sub:26} and let's assume it now.

In summary, we have the following lemma,
\begin{lem}
  \label{lem:two}
  Let $u$ be a solution to (\ref{eqn:smallone}) given by Lemma \ref{lem:one}. In addition to the assumption in Lemma \ref{lem:one}, if $\tilde{\triangle} u_0, a(x,0)$ and $f(x,0)$ are in $\mathcal E^{2,\alpha}$, $\partial_t a$ and $\partial_t f$ are in $\mathcal P^{0,\alpha,T}$ and 
  if we assume further that $u_0$ has finite Dirichlet energy, then we have
  \begin{equation*}
    \norm{\partial_t u}_{\mathcal P^{2,\alpha,T}}\leq C
  \end{equation*}
  where $C$ depends on $T$, $c$ and the all the norms mentioned above.
\end{lem}

\begin{proof}
	It follows from $\int_{S\setminus \set{p_i}} \abs{\nabla u_0}^2 d\tilde{V} < \infty$ and an argment in Lemma \ref{lem:basic} that the right hand side of (\ref{eqn:appendix}) goes to zero (if we choose $S_k$ properly). Then we apply the argument of Lemma \ref{lem:one} to obtain uniform estimates for $w_k$. The proof is done since $\lim_{k\to \infty} w_k= \partial_t u$.
\end{proof}

There is also a version of Lemma \ref{lem:energy} for $\pfrac{u}{t}$.
\begin{lem}
	\label{lem:energytwo}
	In addition to all the assumptions of Lemma \ref{lem:two}, if $\tilde{\triangle} u_0$, $a(x,0)$, $f(x,0)$, $\partial_t a$ and $\partial_t f$ have uniformly bounded Dirichlet energy, then
	{
	\begin{eqnarray*}
	  &&\int_{S\setminus \set{p_i}} \abs{\tilde{\nabla} \partial_t u}^2(t) d\tilde{V}\\
	  &\leq& e^{C_1 t} \int_{S\setminus \set{p_i}} \abs{\tilde{\nabla} (a(x,0)\tilde{\triangle}u_0+f(x,0))}^2 d\tilde{V} \\
	  && +\int_0^t e^{C_1(t-s)}\left( C_2\int_{S\setminus \set{p_i}} \abs{\tilde{\nabla} a}^2 +\abs{\tilde{\nabla} \partial_t a}^2 + \abs{\tilde{\nabla} f}^2 +\abs{\tilde{\nabla} \partial_t f}^2 d\tilde{V}\right) ds.
	\end{eqnarray*}
}
	Here $C_1$ depends on $C^0$ norm of $\partial_t a$ and $c$ in Lemma \ref{lem:two} and $C_2$ depends on the $C^0$ norm of $\partial_t u$, $c$, $\partial_t a$ and $f$. In particular, $\partial_t u$ has bounded energy.
\end{lem}

{\begin{rem}
	As suggested by the referee, we make some comparison between the linear estimates obtained here with Proposition 3.5, 3.6 and 3.7 in \cite{MRS}. To be the best of our knowledge, the parabolic estiamtes proved in this section is more or less the same as the Proposition 3.6 in \cite{MRS}. Proposition 3.7 therein is strictly stronger, because it contains nondegenerate estimate in time time at the cone point. For the same reason, the estimates here are also weaker than those obtained in \cite{new}. However, we find that this not necessary for our proof. We find a work-around as introduced in Section \ref{sec:intro}. The details are in Section \ref{sec:nonlinear}.
\end{rem}
}

\subsection{Maximum principle}
In this section we prove a version of the maximum principle for metrics with conical singularities. It should be noted that a maximum principle was used in the proof of Lemma \ref{lem:one}, see Proposition 3.2 and Corollary 3.5 in \cite{mine}. That one applies only to functions constructed from a specific process. The point here is that we use some regularity assumption instead of knowing where the functions come from.

\begin{rem}
  The assumption being used here, i.e. being bounded and having finite Dirichlet energy, is absolutely not necessary for the validity of integration by parts and maximum principle, as can be seen from the proof below. One can take different approaches to deal with the parabolic equations on a conical surfaces. It turns out that our approach favors the assumption. 

  One may want to compare with Jeffres' idea in \cite{Jeff}, {its elliptic version in Section 5 of \cite{big3}} and its parabolic version, Lemma 11.4 in \cite{new}.
\end{rem}

The following basic lemma shows that this assumption justifies the integration by parts, which is important to the proof of the maximum principle of this subsection.

\begin{lem}
	\label{lem:basic}
	Assume that $u$ and $v$ are bounded functions and have bounded Dirichlet energy and $u$ is in $\mathcal E^{2,\alpha}$, then
	\begin{equation*}
		\int_{S\setminus \set{p_i}} {\triangle} u=0 \quad \mbox{and} \quad \int_{S\setminus \set{p_i}} v\cdot{\triangle u} =-\int_{S\setminus \set{p_i}} \nabla u\cdot \nabla v 
	\end{equation*}
\end{lem}
\begin{proof}
	Note that these two equalities are invariant under conformal change of the metric, hence we may assume the metric is $\tilde{g}$. Since the first inequality follows from the second one by letting $v\equiv 1$, we prove the second equality only.

  For simplicity, assume that there is only one singularity. Near the singularity,
  \begin{equation*}
    \int_0^\delta \int_{\partial B(r)} \abs{\tilde{\nabla} u}^2 d\sigma dr <\infty.
  \end{equation*}
  Here $B(r)$ is the ball of radius $r$ centered at the singularity measured with respect to the cone metric $\tilde{g}$.
  For any $\varepsilon>0$, we claim that there is a sequence $r_i$ such that
  \begin{equation*}
    \int_{\partial B(r_i)} \abs{\tilde{\nabla} u}^2 d\sigma \leq \frac{\varepsilon}{r_i}.
  \end{equation*}
  If the claim is not true, then there exists some $\varepsilon>0$ such that for any $r$,
  \begin{equation*}
	  \int_{\partial B(r)} \abs{\tilde{\nabla} u}^2 d\sigma>\frac{\varepsilon}{r},
  \end{equation*}
  which is a contradiction to the finiteness of the Dirichlet energy.
It suffices to show that
  \begin{equation*}
    \lim_{i\to \infty} \int_{\partial B(r_i)} \abs{\tilde{\nabla} u}d\sigma =0,
  \end{equation*}
  because the boundary term in the integration by parts is $\frac{\partial u}{\partial \nu}v$ and $v$ is bounded. The H\"older inequality implies
  \begin{equation*}
    \int_{\partial B(r_i)} \abs{\tilde{\nabla} u}d\sigma \leq C \left( \int_{\partial B(r_i)} \abs{\tilde{\nabla} u}^2 d\sigma\right)^{1/2} r_i^{1/2}<\varepsilon .
  \end{equation*}
  Since $\varepsilon$ is arbitrary, we are done.
\end{proof}

The following are two maximum principles. The proofs are almost the same, but they are prepared for different purposes.

\begin{lem}
	\label{lem:maximum} Let $u(x,t)$ be a bounded function in $\mathcal P^{2,\alpha,T}$ with bounded Dirichlet energy satisfying
  \begin{equation*}
	  \partial_t u\leq a(x,t) \tilde{\triangle} u(x,t) + F(u),
  \end{equation*}
  where $F$ is a smooth function.
  We assume that $\partial_t a$, $a$ and $a^{-1}$ are bounded. If $C_0=\max_{S\setminus \set{p_i}} {u(\cdot,0)}$ and $h(t)$ is the solution of ODE,
  \begin{equation*}
    \frac{dh}{dt}=F(h)\quad \mbox{and }\quad h(0)=C_0,
  \end{equation*}
  then
  \begin{equation*}
    \max_{S\setminus \set{p_i}} {u(\cdot,t)}\leq h(t).
  \end{equation*}
\end{lem}

\begin{proof}
  For any function $w$, we will denote by $w^+$ the positive part of $w$ given by $\max\{w,0\}$. 
  \begin{eqnarray*}
	  &&\frac{d}{dt}\int_{S\setminus \set{p_i}} [(u(x,t)-h(t))^+]^2 a^{-1} d\tilde{V}\\
	  &=& \int_{S\setminus \set{p_i}} 2 (u-h)^+ \frac{d}{dt}(u-h) a^{-1} d\tilde{V} \\
	  && +\int_{S\setminus \set{p_i}} [ (u-h)^+]^2 \partial_t(a^{-1}) d\tilde{V} \\
	  &\leq & \int_{S\setminus \set{p_i}} 2 (u-h)^+ [\tilde{\triangle} (u-h)+ a^{-1}( F(u)- F(h))] d\tilde{V} \\
	  && +C \int_{S\setminus \set{p_i}} [ (u-h)^+]^2 a^{-1}  d\tilde{V} \\
	  &=& -2\int_{S\setminus \set{p_i}} \abs{\tilde{\nabla} (u-h)^+}^2 d\tilde{V} + \int_{S\setminus \set{p_i}} 2(u-h)^+ a^{-1}F'(\xi)(u-h) d\tilde{V} \\
	  && + C\int_{S\setminus \set{p_i}} [(u-h)^+]^2 a^{-1} d\tilde{V}\\
	  &\leq& C \int_{S\setminus \set{p_i}} [(u-h)^+]^2a^{-1} d\tilde{V}.
  \end{eqnarray*}
  The integration by parts is justified by applying Lemma \ref{lem:basic} with $u-h$ and $(u-h)^+$ in the place of $u$ and $v$ there.
  When $t=0$, we know $\int_{S\setminus \set{p_i}} [(u-h)^+]^2 d\tilde{V}$ is zero. Hence, it is zero forever by ODE comparison. 
\end{proof}

\begin{lem}\label{lem:c0}
	Let $u$ be a solution to (\ref{eqn:one}) which is bounded and has bounded Dirichlet energy. Assume that $\partial_t a$, $a$ and $a^{-1}$ are bounded. If $\tilde{\triangle} u_0$ is bounded, then we have
  \begin{equation*}
    \norm{u(\cdot,t)-u_0}_{C^0}\leq e^{C_1 t}\int_0^t e^{-C_1 s} C_2 ds, 
  \end{equation*}
  where $C_1$ is the $C^0$ norm of $b$ and $C_2$ is the $C_0$ norm of $f+ bu_0+a\tilde{\triangle} u_0$.
\end{lem}

\begin{proof}
  It suffices to apply the method of proof in Lemma \ref{lem:maximum} to the equation of $u-u_0$ 
  \begin{equation*}
	  \partial_t(u-u_0) = a\tilde{\triangle} (u-u_0) + b(u-u_0) + f+ bu_0 +a\tilde{\triangle} u_0.
  \end{equation*}
  Let $h(t)$ be the solution of ODE 
  \begin{equation}\label{odeh}
    \frac{dh}{dt}=  C_1 h +C_2, \quad h(0)=0.
  \end{equation}
  Here $C_1$ is the $C^0$ norm of $b$ and $C_2$ is the $C_0$ norm of $f+ bu_0+a\tilde{\triangle} u_0$.
  \begin{eqnarray*}
	  \partial_t(u-u_0-h) &=&  a\tilde{\triangle} (u-u_0-h) + b(u-u_0)-C_1 h + f+ bu_0 +a\tilde{\triangle} u_0-C_2  \\
	  &\leq& a\tilde{\triangle} (u-u_0-h) + b(u-u_0)-b h, 
  \end{eqnarray*}
  where in the last step we make use of the fact that $h\geq 0$, which follows from (\ref{odeh}).

The same proof as in Lemma \ref{lem:maximum} shows the $u-u_0-h\leq 0$ for all $t$. The other side of estimate is similar.
\end{proof}

\section{Poisson equation}\label{sec:poisson}
In this section, we discuss Poisson equation 
\begin{equation}
	\tilde{\triangle} v=f
  \label{eqn:Poisson}
\end{equation}
on surfaces with conical singularities. 

There are two main results in this section. The first is about the existence of solution to (\ref{eqn:Poisson}) and the second is about the regularity of $v$. It should be noted that for a general conformal metric $g=e^{2u}\tilde{g}$ and the Poisson equation $\triangle v=f$, we can always move the conformal factor to the right hand side to reduce it to (\ref{eqn:Poisson}). This method applies to both the existence problem and the regularity issue.

\begin{thm}
  \label{thm:poisson1}
For any $f$ satisfying
  \begin{equation*}
	  \int_{S\setminus \set{p_i}} f^2 d\tilde{V} \leq C\quad \mbox{and} \quad \int_{S\setminus \set{p_i}} f d\tilde{V} =0,
  \end{equation*}
  there is a solution $v$ to (\ref{eqn:Poisson}). Moreover, $v$ is bounded with bounded Dirichlet energy. Among all bounded functions with bounded Dirichlet energy, the solution is unique up to a constant.
\end{thm}

\begin{proof}
	Let $g_s$ be the metric compatible with the conformal structure of $S$ such that $\tilde{g}=w g_s$ near the cone tip and 
  \begin{equation*}
	  w=r^{2\beta},
  \end{equation*}
  where $r$ is the distance to the singular point with respect to $g_s$.
  
  We then consider the equation
  \begin{equation*}
	  \triangle_{g_s} v= w f,
  \end{equation*}
  which is equivalent to (\ref{eqn:Poisson}).
  By our assumption, we have
  \begin{equation*}
	  \int_{S\setminus \set{p_i}} w f dV_{g_s}=0.
  \end{equation*}
  Therefore, there is a solution $v$ to (\ref{eqn:Poisson}).
  Since $\beta>-1$, there exists some $q>2$ such that
  \begin{equation*}
	  w^{1/2}=r^{\beta} \in L^q(g_s).
  \end{equation*}
  Hence, we have $ w f\in L^{\frac{2q}{2+q}}(g_s)$ because $ w^{1/2} f\in L^2(g_s)$.
  The boundedness of $v$ follows from Sobolev embedding, $L^p$ estimate and the fact $\frac{2q}{2+q}>1$. For the same reason,
  \begin{equation*}
	  \int_{S\setminus \set{p_i}} \abs{\tilde{\nabla} v}^2 d\tilde{V}=\int_{S\setminus \set{p_i}} \abs{\nabla_{g_s} v}^2 dV_{g_s}\leq C.
  \end{equation*}
  For the uniqueness, it suffices to show the only bounded harmonic functions with bounded Dirichlet energy are constants. This is true by integration by parts (Lemma \ref{lem:basic}). 
\end{proof}

The rest of this section is devoted to the regularity issue of the Poisson equation. {We note that the results presented here are known and they should follow rather easily from the theory of edge calculus developed in \cite{Melrose} and \cite{Mazzeo} (see also \cite{big3}). In what follows, we give a self-contained and elementary proof.} The results are useful when we try to apply the maximum principle to $\abs{\nabla f}^2$ for $\triangle f=R-r$.

For simplicity, assume that there is only one conical singularity on $S$. Recall that $\tilde{g}$ is a fixed reference metric which is flat near the cone tip.
\begin{lem}
	\label{lem:regularity2}
	Assume that $\beta<0$. Let $v$ be a solution to
	\begin{equation*}
		\tilde{\triangle}v=f.
	\end{equation*}
	If $v$ is bounded with bounded Dirichlet energy and $f$ is bounded, then $\abs{\tilde{\nabla} v}^2$ is bounded and has bounded Dirichlet energy.
\end{lem}

{\begin{rem}
	The result of this lemma is not true for $\beta>0$. For example, consider the cone metric $\tilde{g}=r^{2\beta} (dx^2+dy^2)$ with $\beta>0$. The function $f(x,y)=x$ is harmonic with bounded Dirichlet energy (check this by using the trivial metric which is conformal to $\tilde{g}$). However, 
	\begin{equation*}
		\abs{\tilde{\nabla} v}^2_{\tilde{g}}= r^{-2\beta},
	\end{equation*}
	which is not bounded for $\beta>0$.
\end{rem}
}

\begin{proof}
In a neighborhood $U$ of the singular point, we have local coordinates so that
\begin{equation*}
	\tilde{g}= r^{2\beta} (dr^2 +r^2 d\theta^2) = r^{2\beta} (dx^2+dy^2).
\end{equation*}
As above, we denote the flat metric $dx^2+dy^2$ by $g_s$ and write $W^{2,p}(g_s)$ for the Sobolev space with respect to $g_s$.

Assume first that $0>\beta>-\frac{1}{2}$. In this case,
\begin{equation*}
	\triangle_{g_s} v = r^{2\beta} f \in L^q(g_s)
\end{equation*}
for some $q>2$.

We claim that $v\in W^{2,q}(g_s)$. (Note that since $v$ satisfies the above equation away from the origin, the claim does not follow directly from usual $L^p$ estimate.) To see this, let $\bar{v}$ be the usual solution of
\begin{equation*}
	\triangle_{g_s} \bar{v}= r^{2\beta} f
\end{equation*}
with boundary value  $\bar{v}=v$ on $\set{r=1}$. We know $\bar{v}\in W^{2,q}(g_s)$. On the other hand, the difference $\bar{v}-v$ is a harmonic function defined on $\set{0<r<1}$ and vanishes on $\set{r=1}$. Moreover, it is bounded and has bounded Dirichlet energy. Hence it is zero by Lemma \ref{lem:basic} and $v=\bar{v}$.

By the Sobolev embedding, $\partial_x v$ and $\partial_y v$ is bounded. Hence, $\abs{\tilde{\nabla} v}^2$ is bounded because
\begin{equation*}
	\abs{\tilde{\nabla} v}^2 = r^{-2\beta} \left( (\partial_x v)^2+ (\partial_y v)^2 \right).
\end{equation*}
Since the Dirichlet energy is conformal invariant, we may compute with respect to $g_s$. 
\begin{eqnarray*}
	\partial_x (\abs{\tilde{\nabla} v}^2 ) &=&  (-2\beta) r^{-2\beta-1} \frac{x}{r} \left( (\partial_x v)^2+ (\partial_y v)^2 \right) \\
	&& + r^{-2\beta} \left( 2\partial^2 _{xx} v \partial_x v + 2 \partial^2_{xy} v \partial_y v \right).
\end{eqnarray*}
This is square integrable because $\partial v$ is bounded and $\partial^2 v$ is in $L^q$ for $q>2$. This finishes the proof when $0>\beta>-1/2$. 

If $-1<\beta\leq -1/2$, we can find positive integer $m$ and $\beta_0\in (-1/2,0]$ such that $1+\beta_0=m(1+\beta)$. Hence, we may consider a cone of order $\beta_0$, which is $m-$fold cover of the original one. Then the lemma with cone of order $\beta$ follows from the one with cone of order $\beta_0$. Precisely, by setting $\rho=\frac{1}{1+\beta} r^{\beta+1}$, we have
\begin{equation*}
	\tilde{g}= d\rho^2 + (1+\beta)^2 \rho^2 d\theta^2.
\end{equation*}
Consider another cone of order $\beta_0$, whose metric is given by
\begin{equation*}
	\hat{g}= d\rho^2 + (1+\beta_0)^2 \rho^2 d\eta^2.
\end{equation*}
The map $\Psi$ from $(\rho,\eta)$ to $(\rho, m\eta \mbox{ mod} \, 2\pi)$ is an $m-$fold isometric covering. By setting $\hat{v}=v\circ \Psi$ and $\hat{f}=f\circ \Psi$, we have
\begin{equation*}
	\triangle_{\hat{g}} \hat{v}=\hat{f}.
\end{equation*}
Since $\hat{f}$, $\hat{v}$ is bounded and $\hat{v}$ has bounded Dirichlet energy, we know $\abs{\hat{\nabla} \hat{v}}^2$ is bounded and has bounded Dirichlet energy. So is $v$ and the lemma is proved. 
\end{proof}

\begin{cor}
	\label{cor:regular}
	If $g=e^{2u}\tilde{g}$ for some $u$ bounded and with bounded Dirichlet energy, $v$ is a solution to
	\begin{equation*}
		\triangle_g v=f
	\end{equation*}
	for bounded $f$. If $v$ is bounded with bounded Dirichlet energy, so is $\abs{\nabla v}^2$.
\end{cor}
\begin{proof}
	Moving the conformal factor to the right, we have
	\begin{equation*}
		\tilde{\triangle} v= e^{2u} f.
	\end{equation*}
	Lemma \ref{lem:regularity2} implies that $\abs{\tilde{\nabla}v}^2$ is bounded and has bounded Dirichlet energy. The claim follows from our assumption on $u$ and the fact that
	\begin{equation*}
		\abs{\nabla v}^2 = e^{-2u} \abs{\tilde{\nabla}v}^2.
	\end{equation*}
\end{proof}

\section{Local existence}
\label{sec:iteration}
The purpose of this section is to show a local solution to the normalized Ricci flow if the initial metric is good.

\begin{thm}\label{thm:short}
	If $u_0$ and $K_{0}$ are bounded in $\mathcal E^{2,\alpha}$, $\tilde{\triangle} K_{0}$ is bounded and $u_0$ and $K_0$ have finite Dirichlet energy, then there exists some $T>0$ depending on the $\mathcal E^{2,\alpha}$ norms of $u_0$, $K_{u_0}$ and $C^0$ norm of $\tilde{\triangle} K_0$, such that we have a solution $u(x,t), t\in [0,T)$ with $u_0$ as initial value satisfying
  \begin{equation}
	  \pfrac{u}{t}=e^{-2u}\tilde{\triangle} u +\frac{r}{2}-e^{-2u} \tilde{K}.
    \label{eqn:rf}
  \end{equation}
\end{thm}

\begin{proof}
	It follows immediately from the formula $K=e^{-2u}(-\tilde{\triangle} u+\tilde{K})$ that $\tilde{\triangle} u_0$ is also bounded in $\mathcal E^{2,\alpha}$.

  Let $C_0$ be a constant depending only on an upper bound of $\mathcal E^{2,\alpha}$ norms of $u_0$ and $\tilde{\triangle} u_0$ and $T$ be small positive number to be determined later. We shall construct a sequence of $u_i$ satisfying for $i\geq 1$,

  (a)
  \begin{equation*}
	  \norm{u_i-u_0}_{\mathcal P^{0,\alpha,T}} + \norm{\partial_t u_i- \left( e^{-2u_0}\tilde{\triangle} u_0+\frac{r}{2}-e^{-2u_0}\tilde{K} \right)}_{\mathcal P^{0,\alpha,T}}\leq C_0,
  \end{equation*}
  where $C_0$ is larger than $\mathcal E^{2,\alpha}$ norms of $u_0$ and $\tilde{\triangle} u_0$. It follows immediately that
  \begin{equation*}
	  \norm{u_i}_{\mathcal P^{0,\alpha,T}} + \norm{\partial_t u_i}_{\mathcal P^{0,\alpha,T}}\leq C(C_0);
  \end{equation*}

  (b) $u_i\in \mathcal P^{2,\alpha,T}$ satisfies
  \begin{equation}
	    \pfrac{u_i}{t}=e^{-2u_{i-1}}\tilde{\triangle} u_i +\frac{r}{2}-e^{-2u_{i-1}} \tilde{K}
    \label{eqn:linear}
  \end{equation}
  with initial condition
  \begin{equation*}
	  u_i(x,0)=u_0(x);
  \end{equation*}

  (c) For any $t\in [0,T]$, $u_i(t)$ and $\partial_t u_i (t)$ have bouned Dirichlet energy.

  The construction is by induction and obviously $u_0(x,t)=u_0(x)$ satisfies the above (a) and (c). Now assume $u_{i-1}$ is known. We take $u_i$ to be the solution of (\ref{eqn:linear}) with initial value condition $u_i(x,0)=u_0$ given by Lemma \ref{lem:one} so that (b) holds. Lemma \ref{lem:energy} and Lemma \ref{lem:energytwo} implies that (c) holds. It remains to check that the estimates in (a) holds for $u_i$ as long as $T$ is small. This is done in several steps.

  (1) The $\mathcal P^{0,\alpha,T}$ norm of $a=e^{-2u_{i-1}}$ and $f=\frac{r}{2}-e^{-2u_{i-1}}\tilde{K}$ are bounded by $\mathcal P^{0,\alpha,T}$ norm of $u_{i-1}$; $u_0\in \mathcal E^{2,\alpha}$; Hence, Lemma \ref{lem:one} implies that
  \begin{equation*}
	  \norm{u_i}_{\mathcal P^{2,\alpha,T}}\leq C (C_0).
  \end{equation*}

  (2) $a(x,0)$ and $f(x,0)$ depend only on $u_{i-1}(x,0)=u_0(x)$, hence their $\mathcal E^{2,\alpha}$ norms are bounded; by assumption $\mathcal E^{2,\alpha}$ norm of $\tilde{\triangle} u_0$ is also bounded; by (a),
  \begin{equation*}
	  \norm{\partial_t a}_{\mathcal P^{0,\alpha,T}} +\norm{\partial_t f}_{\mathcal P^{0,\alpha,T}}\leq C(C_0).
  \end{equation*}
  Then by Lemma \ref{lem:two}, we have
  \begin{equation*}
	  \norm{\partial_t u_i}_{\mathcal P^{2,\alpha,T}}\leq C(C_0).
  \end{equation*}

  (3) By (1) and (2), 
  \begin{equation*}
	  \norm{u_i-u_0}_{\mathcal P^{2,\alpha,T}}+\norm{\partial_t u_i- \left( e^{-2u_0}\tilde{\triangle} u_0+\frac{r}{2}-e^{-2u_0}\tilde{K} \right)}_{\mathcal P^{2,\alpha,T}}\leq C(C_0).
  \end{equation*}

  (4) Applying Lemma \ref{lem:c0} to the equations of both $u_i$ and $\partial_t u_i$, we have
  \begin{equation*}
	  \norm{u_i-u_0}_{C_0}+\norm{\partial_t u_i- \left( e^{-2u_0}\tilde{\triangle} u_0+\frac{r}{2}-e^{-2u_0}\tilde{K} \right)}_{C_0}\leq e^{C_1 T}\int_0^T e^{-C_1 s} C_2 ds.
  \end{equation*}
  Here $C_1$ depends only on $C_0$ and $C_2$ depends on both $C_0$ and $\tilde{\triangle} K_0$(see remark below). It implies that we can choose $T$ small so that the above is smaller than anything we want.
  \begin{rem} 
	  $e^{-2u_0}\tilde{\triangle} u_0+\frac{r}{2}-e^{-2u_0}\tilde{K}$ is just $\frac{r}{2}-K_{0}$. Lemma \ref{lem:c0} says the $C^0$ norm of $\tilde{\triangle} (\frac{r}{2}-K_{0})$ is involved in the determination of the above constants. In fact, this is the only place where we use the assumption on the bound of $\tilde{\triangle} K_{0}$. 
\end{rem}

(5) We then combine (3) and (4) using the interpolation inequality 
\begin{equation}\label{eqn:inter}
	\norm{w}_{\mathcal P^{0,\alpha,T}}\leq C \norm{w}_{C^0}^{1-\alpha} \norm{w}_{\mathcal P^{2,\alpha,T}}^{\alpha}
\end{equation}
to get (a). (\ref{eqn:inter}) follows from the interpolation of usual H\"older norms
\begin{equation*}
	\norm{w}_{C^\alpha(\Omega\times [0,T])}\leq C \norm{w}^{1-\alpha}_{C^0(\Omega\times [0,T])} \norm{w}^\alpha_{C^{0,1} (\Omega\times [0,T])},
\end{equation*}
a proof of which can be found in the proof of Lemma 3.4 in \cite{mine}.

Once $u_i$ is constructed, it follows from (1) and (2) above that $\mathcal P^{2,\alpha,T}$ norms of $u_i$ are uniformly bounded. This implies that $u_i$ subconverges. We claim if $T$ is chosen to be small, then $\lim_{i\to \infty} u_i$ exists and it is the solution we want. To see this, subtract (\ref{eqn:linear}) for $i$ and $i+1$ to get
\begin{equation*}
	\pfrac{}{t}(u_{i+1}-u_i) = e^{-2 u_{i}} \tilde{\triangle} (u_{i+1}-u_i) + (e^{-2 u_{i}}-e^{-2u_{i-1}})\tilde{\triangle} u_i -(e^{-2 u_{i-1}}- e^{-2u_i})\tilde{K}.
\end{equation*}
By (1) and (2) above, we have
\begin{equation*}
	\abs{(e^{-2 u_{i}}-e^{-2u_{i-1}})\tilde{\triangle} u_i -(e^{-2 u_{i-1}}- e^{-2u_i})\tilde{K}} \leq C(C_0) \norm{u_i-u_{i-1}}_{C^0(S\times [0,T])}.
\end{equation*}
Since $u_i$ and $u_{i+1}$ are bounded and have bounded energy, we can apply Lemma \ref{lem:c0} to see
\begin{equation*}
	\norm{u_{i+1}-u_i}_{C^0(S\times [0,T])}\leq T C(C_0) \norm{u_i-u_{i-1}}_{C^0(S\times [0,T])}.
\end{equation*}
Hence, if we choose $T$ small, then the sequence $u_i$ is Cauchy in $C^0$ norm and hence the limit is unique and the theorem is proved.
\end{proof}

\begin{cor}
  Let $u(t)$ be a local solution to (\ref{eqn:rf}) given above.
  If 
  \begin{equation*}
    r=\frac{4\pi\chi(S,\beta)}{{V_0}},
  \end{equation*}
  then the volume of $g(t)$ is a constant and
  \begin{equation*}
    2\pi \chi(S,\beta)=\int_{S\setminus \set{p_i}} K_t dV_t.
  \end{equation*}
\end{cor}
\begin{proof}
  Since $\tilde{g}$ is flat near the singularities, then by using a version of Gauss-Bonnet theorem involving the geodesic curvature on the boundary, it is easy to show
  \begin{equation*}
    2\pi\chi (S,\beta)=\int_{S\setminus \set{p_i}} \tilde{K} d\tilde{V}.
  \end{equation*}
  Since the Gauss curvature of $e^{2u}\tilde{g}$ is given by $e^{-2u}(-\tilde{\triangle} u+\tilde{K})$, Gauss-Bonnet theorem remains true for $e^{2u}\tilde{g}$ if and only if 
  \begin{equation*}
	  \int_{S\setminus \set{p_i}} \tilde{\triangle} u d\tilde{V}=0,
  \end{equation*}
  which is true for our $u(t)$ (this is a consequence of Lemma \ref{lem:basic} and Proposition \ref{prop:good} below). Denote the volume of $g(t)$ by $V(t)$.
  \begin{equation*}
	  \frac{d}{dt} V(t)= \int_{S\setminus \set{p_i}} 2\tilde{\triangle} u + re^{2u}-2\tilde{K} d\tilde{V}=rV(t)-rV_0.
  \end{equation*}
  Since $V(0)=V_0$, $V(t)\equiv V_0$.
\end{proof}

For the local solution constructed above, we know $u$ and $\partial_t u$ (hence $K$) are bounded. Later, to show the long time existence and convergence, we will apply maximum principles to the Gauss curvature $K$ of $e^{2u} \tilde{g}$. As we have shown before, for that purpose, we need to show the Dirichlet energy of $K$ (or equivalently, $\triangle u$) is uniformly bounded for $x\in S$ and $t\in [0,T]$. Please note that the condition (c) above does not imply a uniform upper bound.

\begin{prop}\label{prop:good}
  For the solution $u$ obtained above, we have
  \begin{equation*}
	  \int_{S\setminus \set{p_i}} \abs{\tilde{\nabla} u}^2 +\abs{\tilde{\nabla} \tilde{\triangle} u}^2 d\tilde{V}\leq C
  \end{equation*}
  for $t\in [0,T]$. Here $C$ is a constant depending on $\mathcal E^{2,\alpha}$ norm of $u_0$ and $K_0$, Dirichlet energy of $u_0$ and $K_0$, the $C^0(S\times [0,T])$ norm of $u(x,t)$ and $T$.
\end{prop}

\begin{rem}
	We note that the Proposition is still true if we replace $\tilde{\nabla}$, $\tilde{\triangle}$ and $d\tilde{V}$, with $\nabla$, $\triangle$ and $dV$ (with respect to some conformal metric $g=e^{2v}\tilde{g}$, if $v$ is bounded and has bounded Dirichlet energy). This follows from two facts: (1) the Dirichlet energy is conformal invariant; (2) $\triangle u= e^{-2v} \tilde{\triangle} u$.
\end{rem}

\begin{proof}
Let $u_i$ be the sequence we used in the iteration process of the previous theorem to obtain $u$. Since $u$ is the limit of $u_i$, it suffices to show that
\begin{equation*}
	\int_{S\setminus \set{p_i}} \abs{\tilde{\nabla} u_i}^2 + \abs{\tilde{\nabla} (\tilde{\triangle} u_i)}^2 dx \leq C
\end{equation*}
uniformly for all $i$ and $t\in [0,T]$.

We apply the derivation of (\ref{eqn:energy}) to (\ref{eqn:linear}) to see that
{\begin{eqnarray*}
	\int_{S\setminus \set{p_i}} \abs{\tilde{\nabla} u_i}^2(t) d\tilde{V}
	&\leq& e^{C_1 t} \int_{S\setminus \set{p_i}} \abs{\tilde{\nabla} u_0}^2 d\tilde{V }+\int_0^t e^{C_1(t-s)} \left( C_2 \int_{S\setminus \set{p_i}} \abs{\tilde{\nabla} u_{i-1}}^2 d\tilde{V}\right) ds. 
\end{eqnarray*}
}

Let $f_i(t)$ be $\int_{S\setminus \set{p_i}} \abs{\tilde{\nabla} u_i}^2 d\tilde{V}$. Then $f_0(t)=\int_{S\setminus \set{p_i}} \abs{\tilde{\nabla} u_0}^2 d\tilde{V}$ is a constant, which we denote by $C_3$. We have
{\begin{equation*}
	f_i(t)\leq e^{C_1 t} C_3 + \int_0^t e^{C_1(t-s)} C_2 f_{i-1}(s) ds .
\end{equation*}
}
We claim that $f_i(t)\leq C_3 e^{(C_1+C_2)t}$ for $t\in [0,T]$. To see this, set $g_0(t)\equiv C_3$ and

{\begin{equation*}
	g_i(t)= e^{C_1 t} C_3 +\int_0^t e^{C_1(t-s)} C_2 g_{i-1}(s) ds,
\end{equation*}
which is equivalent to
\begin{equation*}
	\frac{d}{dt} g_i (t)= C_1 g_i(t) +C_2 g_{i-1}(t).
\end{equation*}
}
It is easy to prove by induction that $f_i(t)\leq g_i(t)$. On the other hand, $g_1(t)\geq g_0(t)$. Again, by induction, we can show that
\begin{equation*}
g_i(t)\geq g_{i-1}(t).
\end{equation*}
Therefore,
\begin{equation*}
\frac{d}{dt} g_i(t)\leq (C_1+C_2) g_i,
\end{equation*}
from which our claim follows. Hence, we have proved
\begin{equation*}
	\int_{S\setminus \set{p_i}} \abs{\tilde{\nabla} u}^2 d\tilde{V} \leq C
\end{equation*}
for $t\in [0,T]$.

For the energy of $\tilde{\triangle} u$, due to the equation (\ref{eqn:rf}), it suffices to consider $\partial_t u$. We take a $t$-derivative of (\ref{eqn:linear}) to get
\begin{eqnarray*}
	\partial_t (\partial_t u_{i}) &=&  e^{-2u_{i-1}}\tilde{\triangle} (\partial_t u_i) -2 e^{-2u_{i-1}} (\partial_t u_{i-1}) (\tilde{\triangle} u_i) +2 e^{-2u_{i-1}}\tilde{K}(\partial_t u_{i-1}) \\
	&=&  e^{-2u_{i-1}}\tilde{\triangle} (\partial_t u_i) -2 (\partial_t u_{i-1}) (\partial_t u_i -\frac{r}{2}+ e^{-2u_{i-1}}\tilde{K}) +2 e^{-2u_{i-1}}\tilde{K}(\partial_t u_{i-1}) \\
	&=& a(x,t) \tilde{\triangle} (\partial_t u_i) + b(x,t) (\partial_t u_i) + f(x,t).
\end{eqnarray*}
Here
\begin{equation*}
b(x,t)=-2 (\partial_t u_{i-1})
\end{equation*}
and
\begin{equation*}
	f(x,t)=-2(\partial_t u_{i-1}) (-\frac{r}{2}+ e^{-2u_{i-1}}\tilde{K}) +2 e^{-2u_{i-1}}\tilde{K} (\partial_t u_{i-1})=r (\partial_t u_{i-1}).
\end{equation*}

We can argue as before to show that 
\begin{equation*}
	\int_{S\setminus \set{p_i}} \abs{\tilde{\nabla} (\partial_t  u_i)}^2 d\tilde{V}
\end{equation*}
is bounded.
We can use (\ref{eqn:linear}) again with the fact that $\int_{S\setminus \set{p_i}} \abs{\tilde{\nabla} u_{i-1}}^2 d\tilde{V}$ are bounded to conclude the second part of our claim.
\end{proof}

To conclude this section, we will prove a uniqueness result to (\ref{eqn:rf}). 
\begin{lem}
  \label{lem:unique}
  If $u_1$ and $u_2$ are two solutions to (\ref{eqn:rf}) with the same initial value and both of them are bounded with bounded Dirichlet energy and $\tilde{\triangle} u_1$ and $\tilde{\triangle} u_2$ are bounded, then they are the same.
\end{lem}
\begin{proof}
  Set $w=u_1-u_2$. Then
  \begin{equation*}
	  \partial_t w=e^{-2u_1}\tilde{\triangle} w + (e^{-2u_1}-e^{-2u_2}) (\tilde{\triangle} u_2-\tilde{K}),
  \end{equation*}
  which is equivalent to 
  \begin{equation*}
	  \partial_t w=e^{-2u_1}\tilde{\triangle} w -2\left( \int_0^1 e^{-2u_2-2t(u_1-u_2)}dt\right) (\tilde{\triangle} u_2-\tilde{K})w .
  \end{equation*}
  Since $w$ is bounded with bounded Dirichlet energy and the coefficients in the above equation are bounded, then maximum principle (Lemma \ref{lem:c0} with $f\equiv 0$ and $u_0\equiv 0$) shows that $w$ is always zero.
\end{proof}

\section{Long time existence and convergence}\label{sec:nonlinear}

The section consists of two subsections. In the first one, we define a notion of maximal solution and then show that the solution exists as long as its curvature remains bounded. In the second one, we discuss the long time existence and convergence of the normalized Ricci flow under various assumptions.

\subsection{Maximal solution}

\begin{defn}
	Let $u$ be a solution to (\ref{eqn:rf}) defined on $[0,T)$ such that both $u$ and $K$ are in $\mathcal P^{2,\alpha,\tilde{T}}$ for any $\tilde{T}<T$ and have finite Dirichlet energy for $t\in [0,T)$. $u$ is said to be maximal if there are no other solutions defined on $[0,T']$ with $T'>T$ which satisfies the same restrictions and agrees with $u$ on  $[0,T)$.
\end{defn}

\begin{lem}\label{lem:long}
	Assume that $u$ is a maximal solution defined above and $\tilde{\triangle} K_0$ has bounded Dirichlet energy. Then either $T$ is infinity, or $T$ is finite and
\begin{equation*}
  \lim_{t\to T} \max_{S\setminus \set{p_i}} \abs{K(x,t)} =\infty.
\end{equation*}
\end{lem}

\begin{proof} Assume without loss of generality that $T\geq 1$.
It suffices to show that if
\begin{equation}\label{eqn:lower}
  \lim_{t\to T} \max_{S\setminus \set{p_i}} \abs{K(x,t)} \leq C_1,
\end{equation}
then the solution can be extended to $[0,T+\delta]$ for some $\delta>0$.
(\ref{eqn:lower}) and the evolution equation for $u$ imply that
\begin{equation*}
  \lim_{t\to T} \max_{S\setminus \set{p_i}} \abs{u(x,t)} \leq C(C_1,T).
\end{equation*}

Since $u$ satisfies (\ref{eqn:rf}) and $K$ satisfies
\begin{equation*}
	\partial_t K = e^{-2u} \tilde{\triangle} K +K(2K-r),
\end{equation*}
we can apply the usual interior estimate for linear parabolic equations to see that the $\mathcal P^{2,\alpha,T}$ norms of both $u$ and $K$ are bounded by $C(C_1,T)$. For simplicity, we concentrate on a neighborhood of a singularity. More precisely, let $p$ be one of $\set{p_i}$ and assume coordinates are defined for geodesic ball $B_{\rho_0}(p)$. For $\rho<\rho_0/2$, we study the above two equations on domain $(B_{2\rho}(p)\setminus B_{\rho}(p))\times [0,T]$. By scaling the domain to $(B_2\setminus B_1)\times [0, \rho^{-2} T]$, we obtain equations
\begin{equation*}
	\rho^{2} K = e^{-2u} ( - \bar{\triangle} u + \rho^2 \tilde{K}),
\end{equation*}
\begin{equation*}
	\pfrac{u}{s}= e^{-2u} \bar{\triangle} u + \rho^2 (\frac{r}{2}- e^{-2u}\tilde{K}) 
\end{equation*}
and
\begin{equation*}
	\pfrac{K}{s}= e^{-2u} \bar{\triangle} K + \rho^2 K(2K-r).
\end{equation*}
Here $\bar{\triangle}$ is the Laplacian on the scaled domain and $s=\rho^{-2} t$. Since $K$ is bounded, we know $\pfrac{u}{t}$ is bounded and so is $\pfrac{u}{s}$. The uniform bound of $u$ and $K$ and the first equation imply that $\bar{\nabla} u$ is bounded on the scaled domain. Hence, $u$ is $C^\alpha$ bounded on domain $(B_2\setminus B_1)\times [0,\rho^{-2}T]$. We need this fact because $u$ appears as the coefficient of the second order derivatives. We then apply $L^p$ and Schauder estimates to the second and third equations to get $C^{2,\alpha}$ estimtes of $u$ and $K$. By definition of $\mathcal P^{2,\alpha,T}$, these are the estimates we need.

We claim that
\begin{equation}\label{eqn:key}
	\lim_{t\to T} \max_{S\setminus \set{p_i}} \abs{\tilde{\triangle} K}\leq C(C_1,T).
\end{equation}

To see this, let's consider the evolution equation of $R$, which is $2K$,
\begin{equation}\label{eqn:R}
	\partial_t R= e^{-2u} \tilde{\triangle} R + R(R-r).
\end{equation}

Before going into the detail, we outline the idea of the proof. First, notice that it is equivalent to have a bound for $\partial_t R$ since $R$ and $u$ are already uniformly bounded. If we take a time derivative of (\ref{eqn:R}), we will see that $\partial_t R$ satisfies a linear parabolic equation. Since we have assumed that $\tilde{\triangle} R_0$ is bounded, it is natural to expect that our claim is true. The technical issue is that we can't apply maximum principle to $\partial_t R$ directly, because we don't know if it is bounded and has bounded Dirichlet energy. The idea is to construct another solution $\tilde{R}$ to (\ref{eqn:R}) by approximation so that (1) both $R$ and $\tilde{R}$ are bounded with bounded energy, which implies that they are the same function and (2) $\partial_t \tilde{R}$ is bounded as a consequence of the construction. The proof is carried out in several steps.

{\bf Step 1.} 
Let $S_k$ be as before. Consider the solution $R_k$ to the following equation.
\begin{equation}\label{eqn:rk}
  \left\{
  \begin{array}[]{l}
	  \partial_t R_k=e^{-2u} \tilde{\triangle} R_k + R_k(R_k-r) \\
    \pfrac{R_k}{n}|_{\partial S_k}=0 \\
    R_k(x,0)=2K_0.
  \end{array}
  \right.
\end{equation}
Here $K_0$ is the initial value of Gauss curvature. 
(\ref{eqn:rk}) is a nonlinear parabolic equation with Neumann boundary condition on a smooth manifold with boundary. 

We claim:
\begin{lem}
	\label{lem:wehaverk}
	There is a constant $T'$ depending only on $\norm{K_0}_{C^0(S\setminus \set{p_i})}$ and a solution $R_k$ to (\ref{eqn:rk}) defined on $S_k\times [0,T']$ such that
	\begin{equation*}
		\sup_{t\in [0,T']}\norm{R_k}_{C^0(S_k)}+ \norm{\tilde{\nabla} R_k}_{L^2(S_k)}\leq C(T').
	\end{equation*}
	Here $C(T')$ is independent of $k$.
\end{lem}

The proof of this lemma is made complicated by the compatibility condition and is put in the Appendix.

	{\bf Step 2.} For any compact set $W\subset S\setminus \set{p_i}$ and $k$ sufficiently large, we have uniform $C^{2,\alpha}$ estimates of $R_k$ on $W\times [0,T']$. This allows us to assume $\lim_{k\to \infty} R_k=\tilde{R}$. The same argument as at the beginning of the proof turns the uniform $C^0$ estimates of $R_k$ over $S_k\times [0,T']$ into the $\mathcal P^{2,\alpha,T'}$ estimate of $\tilde{R}$.

	Moreover, by Lemma \ref{lem:wehaverk},
	\begin{equation*}
		\sup_{t\in [0,T']}\norm{\tilde{R}}_{C^0(S\setminus \set{p_i})}+ \norm{\tilde{\nabla} \tilde{R}}_{L^2(S\setminus \set{p_i})}\leq C(T').
	\end{equation*}


Since both $R$ and $\tilde{R}$ satisfy the same equation and both of them are bounded with bounded Dirichlet energy, we can show that $R\equiv \tilde{R}$ for $t\in [0,T']$ by subtracting the two equations and applying maximum principle to $R-\tilde{R}$.

{\bf Step 3.} 
If we take the $t$-derivative of (\ref{eqn:rk}), we get
\begin{eqnarray*}
	  \partial_t(\partial_t R_k)= e^{-2u} \tilde{\triangle} (\partial_t R_k )+ (\partial_t R_k) (2R_k-r-2\partial_t u) 
     +2\partial_t u  R_k(R_k-r)
\end{eqnarray*}
Due to the possible failure of compatibility condition at the corner, we do not expect $\pfrac{R_k}{t}$ to be bounded. However, Theorem \ref{app:good} implies the existence of $w_k$ satisfying
\begin{equation*}
	\left\{
		\begin{array}[]{l}
	  \partial_t w_k = e^{-2u} \tilde{\triangle} w_k + w_k (2R_k-r-2\partial_t u) 
     +2\partial_t u  R_k(R_k-r)\\
     \pfrac{w_k}{\nu}|_{\partial S_k} =0 \\
     w_k|_{t=0}= e^{-2u_0} \tilde{\triangle} R_0 + R_0(R_0-r).
     \end{array}
     \right.
\end{equation*}
The initial and boundary conditions are the same as those of $\partial_t R_k$.

As before, Lemma \ref{lem:ut} implies that
\begin{equation*}
	\norm{\partial_t R_k -w_k}_{L^1(S_k\times [0,T'])}\leq C \norm{\partial_\nu (e^{-2u_0} \tilde{\triangle} R_0 + R_0(R_0-r))}_{L^1(\partial S_k)}.
\end{equation*}
When $k\to \infty$, the right hand side vanishes because of the assumption that $\tilde{\triangle} K_0$ has bounded energy and the argument in Lemma \ref{lem:basic}. This implies 
\begin{equation*}
	\lim_{k\to \infty} \partial_t R_k =\lim_{k\to \infty} w_k.
\end{equation*}

By Remark \ref{rem:maximum}, the maximum principle implies that
\begin{equation*}
  \abs{w_k}(x,t) \leq C
\end{equation*}
for $t\in [0,T']$. Moreover, Lemma \ref{lem:energy} implies that $w_k$ has bounded energy for $T\in [0,T']$.

By passing the limit $k\to \infty$, we obtain that $\partial_t R$ is bounded with bounded energy for $t\in [0,T']$. 

$T'$ may be smaller than $T$, but we note that $T'$ depends only on $C^0$ norm of $R_0$. Therefore we then take $R(x,T')$ as initial value in (\ref{eqn:rk}) and repeat the above argument (Step 1-3) to show
\begin{equation*}
  \abs{\partial_t R}(x,t)\leq C(T)
\end{equation*}
for all $t\in [0,T)$. This finishes the proof of (\ref{eqn:key}).

	In summary, we have proved that for any $0\leq t< T$, $\mathcal E^{2,\alpha}$ norm of $u$ and $K$, $C^0$ norm of $\tilde{\triangle} K$ are bounded by $C(T)$. Moreover, by Proposition \ref{prop:good}, the Dirichlet energy of $u$ and $K$ are bounded. Hence, Theorem \ref{thm:short} implies that for any $t_1<T$, we can solve (\ref{eqn:rf}) with $u(x,t_1)$ as the initial value for a solution defined on $[t_1,t_1+\delta(T)]$. We may choose $t_1$ such that $t_1+\delta(T)> T$. Due to Lemma \ref{lem:unique}, we know the solution on $[t_1,t_1+\delta(T)]$ is an extension of the maximal solution, which is a contradiction. 
\end{proof}

\subsection{Long time existence and convergence}
For the solution we have constructed (by iteration), $K$ is bounded and has finite energy so that the maximum principle applies.

Due to the evolution equation of $K$,
\begin{equation*}
	\partial_t K = e^{-2u}\tilde{\triangle} K +K (2K-r),
\end{equation*}
we may prove the following lemma, which is the first statement of Theorem \ref{thm:main}.
\begin{lem}
	Let $u$ be the solution to (\ref{eqn:rf}) given by Theorem \ref{thm:short}. Then

(1) The Gauss curvature has a lower bound as long as the solution exits;

(2) If $K_0<c<0$, then the Gauss curvature will converge to a negative number exponentially fast. 
\end{lem}

\begin{proof}
	The proof uses Lemma \ref{lem:maximum}. Note that since we have equality instead of inequality, we can apply Lemma \ref{lem:maximum} to both $K$ and $-K$ to get both upper bound and lower bound.

	For (1), we apply Lemma \ref{lem:maximum} to  	
	\begin{equation*}
		\partial_t (-K) = e^{-2u} \tilde{\triangle}(-K) + (-K)(-2(-K)-r)
	\end{equation*}
	with $u=-K$ and $h$ the solution of the ODE
	\begin{equation*}
		\frac{d}{dt}h= h(-2h-r),\qquad h(0)=\max \{\max\set{0,-r/2},\max_{S\setminus \set{p_i}} (-K_0)\}.
	\end{equation*}
	It is not hard to know that $h(t)$ is bounded from above by $h(0)$ as long as the solution exists and Lemma \ref{lem:maximum} implies that so is $-K$.

	(2) Since the initial $K_0$ is negative, we have $r<0$. Set
	\begin{equation*}
		\frac{d}{dt}h_l= h_l (-2h_l -r),\qquad h_l(0)= \max_{S\setminus \set{p_i}} (-K_0) > -r/2
	\end{equation*}
	and
	\begin{equation*}
		\frac{d}{dt}h_u= h_u(2h_u-r),\qquad h_u(0)= \max_{S\setminus \set{p_i}} (K_0)\in [r/2,0).
	\end{equation*}
	Here we used the fact that $r/2$ is the average of $K_0$.
	We then compare $h_l$ with $-K$ and $h_u$ with $K$ as above to see
	\begin{equation*}
		-h_l\leq K \leq h_u
	\end{equation*}
	as long as the solution exists. It is elementry ODE argument that both $h_l$ and $h_u$ converge exponentially fast to $r/2$.
\end{proof}

If the initial curvature is not negative, we can only deal with the sharp cone case ($\beta_i<0$), for a reason which will be clear soon.

By Theorem \ref{thm:poisson1}, we can solve the Poisson equation
\begin{equation*}
  \triangle_{u_0} f_0=R_0-r
\end{equation*}
and $f_0$ is bounded with bounded Dirichlet energy.
Lemma \ref{lem:one} allows us to solve
\begin{equation}\label{eqn:potential}
	\pfrac{f}{t}=e^{-2u}\tilde{\triangle} f +rf
\end{equation}
with initial condition $f(0)=f_0$.
Lemma \ref{lem:energy} implies that $f$ has bounded Dirichlet energy. Thanks to Lemma \ref{lem:two} and Lemma \ref{lem:energytwo}, $\partial_t f$ is also bounded with bounded energy.

Now we claim that for $t>0$
\begin{equation}\label{eqn:wecheck}
	e^{-2u}\tilde{\triangle} f = R-r.
\end{equation}

The above is true for $t=0$ by definition of $f_0$. Writing $\triangle_t=e^{-2u}\tilde{\triangle}$, we obtain by computation
\begin{eqnarray*}
  \partial_t ( \triangle_t f-R+r) &=& (R-r)\triangle_t f +\triangle_t (\partial_t f) -\partial_t R \\
  &=& (R-r) R +(R-r)(\triangle_t f-R) + \triangle_t (\triangle_t f +r f) - \triangle_t R - R(R-r) \\
  &=& \triangle_t (\triangle_t f -R) +r\triangle_t f +(R-r)(\triangle_t f-R) \\
  &=& \triangle_t (\triangle_t f-R+r) +R(\triangle_tf -R +r).
\end{eqnarray*}
As discussed before $\triangle_t f-R+r$ is a bounded function with bounded Dirichlet energy (since $\partial_t f$ is), it satisfies the above equation with zero initial value. Hence, it is zero as long as it exists.

Following Hamilton, we set 
\begin{equation*}
	H=R-r+\abs{\nabla f}^2
\end{equation*}
and compute
\begin{equation*}
	\pfrac{}{t}H=\triangle H-2\abs{M}^2+rH,
\end{equation*}
where $M=\nabla\nabla f-\frac{1}{2}\triangle f\cdot g $.
Therefore, $H$ is a subsolution to
\begin{equation}\label{eqn:H}
	\pfrac{H}{t}\leq \triangle H +rH.
\end{equation}
{By (\ref{eqn:wecheck}) and Lemma \ref{lem:regularity2}, $\abs{\nabla f}^2$ is bounded with bounded energy and so is $R$ by construction in Section \ref{sec:iteration}.} Hence, $H$ is bounded and has bounded energy, then maximum principle can be applied and we know that 

(1) For any $T>0$, we have an upper bound of $R$ depending on $T$. This together with Lemma \ref{lem:long} implies the second statement of Theorem \ref{thm:main}. 

(2) If the Euler number is smaller than zero ($r<0$), then the curvature will become negative everywhere after some time and the normalized Ricci flow in this case will converge to constant curvature metric. {In fact, by (\ref{eqn:H}) and the maximum principle, the maximum of $H=R-r+\abs{\nabla f}^2$ decays exponentially with time $t$. There is positive time $T>0$ such that for $t>T$, $H<-\frac{r}{2}$, which means $R<\frac{r}{2}<0$.} 

Hence to complete the proof of Theorem \ref{thm:main}, it remains to consider the case $\chi=0$. The proof is somewhat different from known ones, which can not be applied directly due to technical reasons. We will first establish some uniform estimates for all $t>0$, which imply a sequentially convergence. We then show that the limit is a flat cone metric and the limit is unique. Hence, the flow converges to this limit as $t\to \infty$.

In the remaining part of the paper, we assume that $r=0$ and let $u(x,t)$ be the solution to (\ref{eqn:rf}). For simplicity, we omit this assumption from the statement of lemmas. 
\begin{lem}
There is some constant $C$ independent of $t$ such that
\begin{equation*}
	\abs{u}+ \abs{K} + \int_{S\setminus \set{p_i}} \abs{\nabla u}^2 dV \leq C.
\end{equation*}
\end{lem}

\begin{proof}
	Applying the maximum principle to (\ref{eqn:H}), we obtain that $K$ is uniformly bounded. For $u$, consider the evolution equation (\ref{eqn:potential}) of the potential function $f$. Since $r=0$, we have uniform bound for $\abs{f}$.
	The flow equation can be written as
\begin{equation*}
	\partial_t g= (- e^{2u} \tilde{\triangle} f) g =(- \partial_t f) g.
\end{equation*}
Integrating over $t$, we have
\begin{equation*}
	g(t)= e^{f_0(x)- f(x,t)} g(0).
\end{equation*}
This provides the uniform bound for $\abs{u}$. For the energy of $u$, note that $\triangle u$ is uniformly bounded since $K=e^{-2u}(-\tilde{\triangle} u+ \tilde{K})$ and due to Lemma \ref{lem:basic}
\begin{equation*}
	\int_{S\setminus \set{p_i}} \abs{\nabla u}^2 dV = -\int_{S\setminus \set{p_i}} u \triangle u dV.
\end{equation*}

\end{proof}

With this uniform estimate, we know that for any $t_i$ going to $+\infty$, there is a subsequence (still denoted by $t_i$) such that $u(\cdot,t_i)$ converges to some $u_\infty$. By interior estimates of (\ref{eqn:rf}), the convergence is smooth away from the singularity. Moreover, $u_\infty$ is bounded and has bounded energy. For future use, we note that the volume 
\begin{equation*}
	\int_{S\setminus \set{p_i}} e^{2u_\infty} d\tilde{V}=\lim_{i\to \infty} \int_{S\setminus \set{p_i}} e^{2u_i} d\tilde{V},
\end{equation*}
because $u_i$ is uniformly bounded and the convergence is uniform on any compact set away from $\set{p_i}$.

Next, we prove the Gauss curvature of the limit is zero. It follows from Proposition 5.29 of \cite{DK}.

\begin{lem}\label{lem:dk}
	There is a constant depending on the initial value such that
	\begin{equation*}
		\sup_{x\in {S\setminus \set{p_i}}} \abs{\nabla f}^2 \leq \frac{C}{1+t}.
	\end{equation*}
\end{lem}
\begin{proof}
	Exactly the same proof. It suffices to note that we can apply the maximum principle to $t\abs{\nabla f}^2 +f^2$ for reason discussed above.
\end{proof}

Let $\Omega$ be some domain of ${S\setminus \set{p_i}}$ away from the singularity. For any $t_i$, set $v_i(x,t)= u(x,t-t_i)$ and $f_i(x,t)=f(x,t-t_i)$. Then $v_i$ is a solution to (\ref{eqn:rf}) on $\Omega\times [-1,0]$ and $f_i$ is still the potential function in the sense that if $R_i$ is the scalar curvature of $e^{2v_i} g_0$,
\begin{equation*}
	e^{-2v_i} \tilde{\triangle} f_i=R_i
\end{equation*} 
and $f_i$ solves
\begin{equation*}
	\partial_t f_i = e^{-2v_i} \tilde{\triangle} f_i.
\end{equation*}
Moreover, we have $C^k$ uniform bound of $v_i$ and $f_i$ over $\Omega' \times [-1/2,0]$ for some smaller domain $\Omega'$. Passing to the limit, denote the limit by $v_\infty$ and $f_\infty$ respectively. Thanks to Lemma \ref{lem:dk}, $f_\infty$ is constant in space. The limit equation of $f_\infty$ then implies that $f_\infty$ is constant in $\Omega'\times [-1/2,0]$.  Hence, the limit metric $e^{2v_\infty} g_0$ has constant curvature zero.

Finally, if we have two sequence of $t_i$ going to infinity and obtain two limit metrics $e^{2u_1}g_0$ and $e^{2u_2}g_0$, we need to show that they are the same. Note that both metrics are flat away from the singularity and both $u_1$ and $u_2$ are bounded with bounded energy. Hence,
\begin{equation*}
	-\tilde{K}= -\tilde{\triangle} u_1=-\tilde{\triangle} u_2.
\end{equation*}
The difference is a constant and the constant must be zero, since the two limit metrics have the same volume. The proof to Theorem \ref{thm:main} is done.

\appendix

\section{}
We discuss in this appendix some results related to the failure of compatibility condition for a parabolic equation with Neumann boundary condition. These results are used several times in the proof of this paper. However, it is not easy to find references for them. For completeness, we prove them here.

To be precise, we consider on a manifold with boundary $M$ the following parabolic equation
\begin{equation}
	\pfrac{u}{t}= a(x,t) \triangle u + b(x,t) u + f
	\label{app:linear}
\end{equation}
with initial condition
\begin{equation*}
	u(0)= u_0\qquad \mbox{on } M
\end{equation*}
and boundary condition
\begin{equation*}
	\pfrac{u}{\nu}|_{\partial M}=0.
\end{equation*}

By the compatibility condition, we mean 
\begin{equation}\label{app:compatible}
	\pfrac{u_0}{\nu}|_{\partial M}=0.
\end{equation}
Obviously, this is a necessary condition for the solution $u$ to be $C^1$ at the corner $\partial M\times \set{0}$. When this fails, we should be very careful about the regularity of solutions to (\ref{app:linear}). In this appendix, we discuss several aspects of this equation.

\subsection{Solvability and Schauder estimates}
Assuming (\ref{app:compatible}), we have the usual solvability and regularity result,
\begin{thm}[Schauder estimate with compatibility condition]\label{thm:schauder}
	Let $a,b,f$ be functions in $C^\alpha(M\times [0,T])$ for some $T>0$ and assume that $0<\lambda <a(x,t)<\lambda ^{-1}$. If $u_0$ in $C^{2,\alpha}(M)$ satisfies (\ref{app:compatible}), then there exists $u$ solving (\ref{app:linear}) with $\pfrac{u}{\nu}|_{\partial M}=0$ and $u(x,0)=u_0$. Moreover, $u$ is in $C^{2,\alpha}(M\times [0,T])$ and we have
	\begin{equation}
		\norm{u}_{C^{2,\alpha}(M\times [0,T])}\leq C \left( \norm{f}_{C^{\alpha}(M\times [0,T])} + \norm{u_0}_{C^{2,\alpha}(M)} \right).
		\label{app:schauder}
	\end{equation}
	Here $C$ depends on  $\lambda$ and $C^\alpha$ norm of $a$ and $b$.
\end{thm}
This theorem is the starting point of our discussion. It is a special case of Theorem 5.18 of \cite{Lie}.

When (\ref{app:compatible}) fails, we can still prove the existence of some solution to (\ref{app:linear}), which is not $C^1$ near the corner. However, we do have some regularity at $t=0$, which is needed for the proof of this paper.

\begin{thm}\label{app:good}
	Let $a,b,f$ and $M$ be as above. For $u_0$ in $C^{2,\alpha}(M)$ not satisfying (\ref{app:compatible}), there is a function $u$ such that

	(1) $u$ is $C^{2,\alpha}$ in $(M\times [0,T])\setminus (\partial M\times \set{0})$ and $u$ satisfies (\ref{app:linear}) pointwisely in this domain;

	(2) $u(x,0)=u_0$ and for any $t>0$, $\pfrac{u}{\nu}|_{\partial M}=0$;

	(3) When $t\to 0$, $u(t)$ converges to $u_0$ in $C^0(M)$ norm;

	(4) If $\int_M \abs{\nabla b}^2 + \abs{\nabla f}^2 dV$ is bounded for $t\in [0,T]$, then when $t\to 0$, we also have $\norm{u(t)-u_0}_{W^{1,2}(M)}\to 0$ .
\end{thm}

The proof of this theorem depends on an approximation of $u_0$. We claim that there exists a sequence of $C^{2,\alpha}(M)$ functions $u_l$ such that
\begin{equation*}
	\lim_{l\to \infty} \norm{u_l-u_0}_{C^0(M)} + \norm{u_l-u_0}_{W^{1,2}(M)}+ \norm{u_l-u_0}_{C^{2,\alpha}(W)}=0
\end{equation*}
for any compact set $W$ in $M\setminus \partial M$ 
and 
\begin{equation*}
	\pfrac{u_l}{\nu}|_{\partial M}=0.
\end{equation*}

To see this, let $\rho$ be the distance to $\partial M$. A neighborhood of $\partial M$ in $M$ is parametrized by coordinates $(\rho,y)$ for $0\leq \rho < \rho_0$ and $y\in \partial M$. For any $\epsilon>0$, we define smooth functions $\eta_\epsilon: [0,\rho_0)\to [0,\rho_0)$ satisfying: (a) $\eta_\epsilon(\rho)=0$ for $0\leq \rho\leq \epsilon/3$; (b) $\eta_\epsilon(\rho)=\rho$ for $\rho> \frac{2}{3}\epsilon$; (c) $\eta_\epsilon'(\rho)\leq 3$; (d) $\abs{\eta_\epsilon''(\rho)}\leq 12 \epsilon^{-1}$.

We then have a family of diffeomorphism $\Psi_\epsilon$, which is identity away from the above mentioned neighborhood of $\partial M$ and is
\begin{equation*}
	\Psi_\epsilon(\rho,y)=(\eta_\epsilon(\rho),y)
\end{equation*}
in it. 

We set $u_l= u_0\circ \Psi_{1/l}$. So near $\partial M$,  we have
\begin{equation*}
	u_l(\rho,y)= u_0 (\eta_{1/l}(\rho),y).
\end{equation*}
We check that $u_l$ is the sequence we need for the claim. Since $u_0$ is in $C^{2,\alpha}(M)$, the only nontrivial part is about the continuity of $W^{1,2}$ norm. For that purpose, we compute
\begin{eqnarray*}
	&& \norm{\nabla u_l(\rho,y)-\nabla u_0(\rho,y)}_{L^2} \\
	&\leq & \norm{\nabla_y u_0 (\eta_{1/l}(\rho),y)-\nabla_y u_0(\rho,y)}_{L^2} + \norm{\partial_\rho u_0(\rho,y)- \partial_\rho u_0(\rho,y) \eta'_{1/l}(\rho)}_{L^2} \\
	&& + \norm{ (\partial_\rho u_0(\rho,y) -\partial_\rho u_0(\eta_{1/l}(\rho),y)) \eta'_{1/l}(\rho)}_{L^2}.
\end{eqnarray*}
When $l\to \infty$, all three terms vanish because $\nabla u_0$ is bounded, the volume of the set on which $\eta_{1/l}(\rho)\ne \rho$ goes to zero and the fact that $\eta'_{1/l}\leq 3$. 

We then apply Theorem \ref{thm:schauder} for each $u_l$ to get a solution $u_l(t)$ to (\ref{app:linear}) with intial value $u_l$. The point is that $u_l$ satisfies the compatibility condition. We can apply the usual Schauder estimates for parabolic equation with Neumann boundary condition on domains away from $\partial M\times \set{0}$ to see that $u_l$ converges to $u$, which satisfies (1) and (2).

To see (3), we subtract the equations of $u_{l_1}$ and $u_{l_2}$ to get
\begin{equation*}
	\pfrac{}{t}(u_{l_1}-u_{l_2})= a(x,t) \triangle (u_{l_1}-u_{l_2}) + b(x,t) (u_{l_1}-u_{l_2})
\end{equation*}
and
\begin{equation*}
	\pfrac{}{\nu} (u_{l_1}-u_{l_2})|_{\partial M}=0.
\end{equation*}
Now the classical maximum principle for parabolic equation with Neumann boundary condition (see the proof of Proposition 3.2 of \cite{mine}) implies that
\begin{equation}\label{eqn:cauchy}
	\norm{u_{l_1}-u_{l_2}}_{C^0(M\times [0,T])}\leq C \norm{u_{l_1}-u_{l_2}}_{C^0(M)}.
\end{equation}
Therefore $u_l$ converges in $C^0(M\times [0,T])$ to $u$ and (3) follows.

(4) is proved by a similar argument with the maximum principle replaced by the computation in Lemma \ref{lem:energy}.

\begin{rem}
	In fact, the proof of Proposition 3.2 of \cite{mine} is not rigorous without taking (3) for granted. The argument we present above fills this gap. Similarly, we have assumed (4) in the proof of Lemma \ref{lem:energy} and we see now that this assumption can be removed.
\end{rem}

\begin{rem}\label{rem:continuous}
Let $u_0'$ be another $C^{2,\alpha}(M)$ function and $u'(t)$ is the solution given by Theorem \ref{app:good}. Since $u$ and $u'$ are obtained as uniform limit of $u_{l_i}$ and $u_{l_i}'$, (\ref{eqn:cauchy}) implies
\begin{equation}\label{eqn:cauchy2}
	\norm{u-u'}_{C^0(M\times [0,T])} \leq C \norm{u_0-u_0'}_{C^0(M)}.
\end{equation}
Making use of this, we can approximate continuous functions on $M$ by $C^{2,\alpha}$ functions so that for each continuous $u_0$, there is a solution $u(t)$ to (\ref{app:linear}) which is in $C^{2,\alpha}(M\times (0,T]) \cap C^0(M\times [0,T])$. It also follows from (\ref{eqn:cauchy2}) the solution is unique in the class of $C^0(M\times [0,T])$.
\end{rem}

\begin{rem}
	\label{rem:maximum}
	If $u$ is a solution to (\ref{app:linear}) with compatible initial function, the a maximum principle as Proposition 3.2 of \cite{mine} holds. If $u$ is a solution provided by Theorem \ref{app:good}, the same maximum principle holds because it is uniform limit of sequences satisfying compatible condition. Moreover, if $u_0$ is just contunious and $u$ is the solution in Remark \ref{rem:continuous}, for the same reason, the maximum principle holds.
\end{rem}

\subsection{Semi-linear equations}

In Lemma \ref{lem:wehaverk}, we claim that there exists a local solution to
\begin{equation}\label{app:rk}
  \left\{
  \begin{array}[]{l}
	  \partial_t R_k=e^{-2u} \tilde{\triangle} R_k + R_k(R_k-r) \\
    \pfrac{R_k}{n}|_{\partial S_k}=0 \\
    R_k(x,0)=2K_0.
  \end{array}
  \right.
\end{equation}
If the compatibility condition
\begin{equation}\label{app:k}
	\pfrac{K_0}{\nu}|_{\partial S_k}=0,
\end{equation}
holds, then this claim follows from the linear estimate (for example, Theorem \ref{thm:schauder}) and a routine iteration method. 

In our case, (\ref{app:k}) may not be true. Hence, we need a workaround by approximation as above. As before, we approximate $K_0$ by a sequence $K_0^l=K_0\circ \Psi_{1/l}$ satisfying 
\begin{equation*}
	\lim_{l\to \infty} \norm{K_0^l-K_0}_{C^0(S_k)} + \norm{K_0^l-K_0}_{W^{1,2}(S_k)}+ \norm{K_0^l-K_0}_{C^{2,\alpha}(W)}=0
\end{equation*}
for any compact set $W$ in $S_k\setminus \partial S_k$ 
and 
\begin{equation*}
	\pfrac{K_0^l}{\nu}|_{\partial S_k}=0.
\end{equation*}
With $K_0^l$ as initial values, we can find solutions $R_k^l$ for (\ref{app:rk}) defined on $[0,T_k^l]$ for some $T_k^l>0$. $R_k^l$ exists as long as it is bounded. 

To get a uniform $C^0$ bound (both uniform in $l$ and $k$), we consider the ODE
\begin{equation*}
	h'(t)= h(h-r), \qquad h(0)=\max \abs{R_0}+1> \max\abs{R_k^l}.
\end{equation*}
Since $R_k^l-h$ satisfies the compatibility condition, we can apply the classical maximum principle to the equation astisfied by $R_k^l-h$ to get
\begin{equation*}
	R_k^l\leq h
\end{equation*}
for $t\in [0,T_l]$. A lower bound follows similarly.

Let $T'>0$ be the positive number such that $\max_{t\in [0,T']} h(t)\leq 4 h(0)+ 1$. Then we have $T_k^l\geq T'$ and the uniform $C^0$ estimates for $R_k^l$.

For the uniform energy estimates, we compute
	\begin{eqnarray*}
	\frac{d}{dt} \int_{S_k} \abs{\tilde{\nabla} {R_k^l} }^2 d\tilde{V} &=& 2\int_{S_k} \tilde{\nabla} (\partial_t R_k^l) \cdot \tilde{\nabla} R_k^l d\tilde{V} \\
	&=& -2\int_{S_k} \partial_t R_k^l \tilde{\triangle} R^l_k d\tilde{V} \\
	&=& -2\int_{S_k} e^{-2u}( \tilde{\triangle} R^l_k)^2  d\tilde{V} + \int_{S_k} \tilde{\nabla} (R^l_k(R^l_k-r))\tilde{\nabla} R^l_k d\tilde{V}\\
	&\leq& C \int_{S_k} \abs{\tilde{\nabla} R_k^l}^2 d\tilde{V}.
\end{eqnarray*}
Since the approximation of $K_0$ by $K_0^l$ is strong in $W^{1,2}$, we can integrate this to get the uniform energy bound.

\subsection{Proof of inequality (\ref{eqn:appendix})}\label{sub:26}
This section is devoted to the proof of inequality (\ref{eqn:appendix}). We are slightly more general by considering $M$ instead of $S_k$ and including a linear term $bu$ in the linear equation (\ref{app:linear}).

Since this is a property of linear parabolic equation when compatibility condition fails, it may be of some general interests and we summarize it in the form of a lemma.

\begin{lem}\label{lem:ut}
	Let $M$ be any manifold with boundary and assume $a,b,f$ and $\partial_t a, \partial_t b, \partial_t f$ are $C^{\alpha}$ functions on $M$, $u_0\in C^{2,\alpha}(M)$ and $0<\lambda<a<\lambda^{-1}$. If $u$ is the solution to (\ref{app:linear}) given by Theorem \ref{app:good} and $v$ is the solution (again by Theorem \ref{app:good}) to the equation
\begin{equation}\label{app:ut}
	\pfrac{v}{t}= a\tilde{\triangle} v + \tilde{b} v+ \tilde{f}
\end{equation}
with $v(0):=v_0=a_0\tilde{\triangle} u_0 + b_0 u_0 + f_0$,
where $a_0,b_0$ and $f_0$ are short for $a(x,0), b(x,0)$ and $f(x,0)$,
\begin{equation*}
	\tilde{b}= \partial_t a \cdot a^{-1} + b
\end{equation*}
and
\begin{equation*}
	\tilde{f}= - \partial_t a a^{-1} bu - \partial_t a^{-1} f + \partial_t b u + \partial_t f,
\end{equation*}
then
\begin{equation*}
	\int_{M\times [0,T]} \abs{\pfrac{u}{t}-v} dV dt \leq C \int_{\partial M} \abs{\pfrac{u_0}{\nu}} dV_{\partial M}.
\end{equation*}

\end{lem}
Obviously, $\pfrac{u}{t}$ satisfies (\ref{app:ut}) with the same initial and boundary value conidtion as $v$.

The following remark is not needed for the proof. However, it explains why we need this lemma.
\begin{rem}
	In general, $\pfrac{u}{t}$ is not bounded near the corner $\partial M\times \set{0}$. To see this, we assume that $\abs{\pfrac{u_0}{\rho}}>c>0$ in a neighborhood of $\partial M$, where $\rho$ is the distance to $\partial M$. If $\partial_t u$, hence $\tilde{\triangle} u$, is uniformly bounded (up to $t=0$), then we apply $L^p$ estimate (any $p>1$) and Sobolev embedding to $u$ on $M$ to see
\begin{equation*}
	\norm{\nabla u(t)}_{C^\alpha(M)}\leq C.
\end{equation*}
When $t$ is small but positive, this is a contradiction because for some $y$ with $d(y,\partial M)^\alpha < \frac{c}{2\norm{\nabla u(t)}}_{C^\alpha(M)}$, we have
\begin{equation*}
	\lim_{t\to 0} \abs{\pfrac{u}{\rho}}(y)\geq c/2,
\end{equation*}
but $\pfrac{u}{\nu}(z,t)=0$ for $z\in \partial M$ is the nearest point on $\partial M$ to $y$ and $t>0$.

This shows that although $\partial_t u$ is a solution to (\ref{app:ut}), we can not obtain good estimates for it due to the failure of compatibility condition.
\end{rem}

\begin{rem}\label{B:ok}
	To begin the proof, we assume that 
	\begin{equation*}
		\pfrac{a}{\nu}|_{\partial M} = \pfrac{b}{\nu}|_{\partial M} =\pfrac{f}{\nu}|_{\partial M}=0.
	\end{equation*}
	The reason is that we only require them to be in $C^\alpha (M\times [0,T])$. We can approximate them in this space by functions satisfying the above condition.
\end{rem}

Recall that $\pfrac{u}{t}$ is the limit of $\partial_t u_l(x,t)$, where $u_l(x,t)$ is the solution to (\ref{app:linear}) with initial condition
\begin{equation*}
	u_l(x,0)= u_0\circ \Psi_{1/l}.
\end{equation*}
We claim that
\begin{equation}\label{B:compatible}
	\partial_\nu \left( a_0 \tilde{\triangle} (u_0\circ \Psi_{1/l}) + b_0 u_0\circ \Psi_{1/l} + f_0 \right) =0
\end{equation}
at $\partial M$. 
By Remark \ref{B:ok}, it sufficies to check
\begin{equation*}
	\partial_\nu \tilde{\triangle} (u_0(\eta_{1/l}(\rho),y))=0
\end{equation*}
for $\rho=0$. This follows from the formula
\begin{equation}
	\tilde{\triangle}= \frac{\partial^2}{ \partial \rho^2} + H \cdot \pfrac{}{\rho} + \triangle_\rho,
	\label{eqn:Laplace}
\end{equation}
where $H$ and $\triangle_\rho$ are the mean curvature and the Laplace of the submanifold defined by $\rho=\mbox{constant}$. As a consequence of the claim, the compatibility condition holds for $\partial_t u_l$, which is a solution of (\ref{app:ut}). 

On the other hand, $v$ is also the limit of a sequence of $v_l$, which solves (\ref{app:ut}) with compatible initial value
\begin{equation*}
	v_l(x,0)= (a_0 \tilde{\triangle }u_0 + b_0 u_0 + f_0) \circ \Psi_{1/l}.
\end{equation*}

Now, since $\partial_t u_l (x,t)- v_l(x,t)$ satisfies (\ref{app:ut}) with the compatibility condition, we claim
\begin{equation}\label{eqn:end}
	\norm{\partial_t u_l - v_l}_{L^1(M\times [0,T])}\leq C \norm{ \partial_t u_l (x,0) - v_l(x,0)}_{L^1(M)} + o(1),
\end{equation}
whose proof is postponed to the end of this section. Here $o(1)$ is an infinitesimal quantity when $l$ goes to infinity.

Since $v_l$ converges to $v$ uniformly and $\partial_t u_l$ converges to $\pfrac{u}{t}$ uniformly on any compact set away from the corner, we know
\begin{equation*}
	\norm{\pfrac{u}{t}-v}_{L^1(M\times [0,T])}\leq C \lim_{l\to \infty} \norm{\partial_t u_l (x,0)- v_l(x,0)}_{L^1(M)}.
\end{equation*}

Obviously,
\begin{equation*}
	\norm{f_0 - f_0\circ \Psi_{1/l}}_{L^1(M)}=o(1)
\end{equation*}
and
\begin{equation*}
	\norm{b_0 u_l  - (b_0 u_0)\circ\Psi_{1/l}}_{L^1(M)}=o(1).
\end{equation*}
Since $\tilde{\triangle} u_0$ is bounded, we have
\begin{eqnarray*}
	&& \norm{a_0 \tilde{\triangle} u_l - (a_0 \tilde{\triangle} u_0)\circ \Psi_{1/l}}_{L^1(M)} \\
	&\leq& o(1) + C \norm{\tilde{\triangle} (u_0\circ \Psi_{1/l}) - (\tilde{\triangle} u_0)\circ \Psi_{1/l}} _{L^1(M)}.
\end{eqnarray*}

Since $u_0$ is in $C^{2,\alpha}(M)$ and $\Psi_{1/l}$ is identity map away from a small neighborhood of $\partial M$, we have
\begin{equation}
	\norm{\triangle_\rho (u_0\circ \Psi_{1/l})- (\triangle_\rho u_0)\circ \Psi_{1/l} }_{L^1(M)}=o(1).
	\label{eqn:good1}
\end{equation}
Similarly, $H$ and $\pfrac{u_0}{\rho}$ are bounded on $M$, which implies
\begin{eqnarray*}
	&&\norm{H\cdot \pfrac{}{\rho} (u_0\circ \Psi_{1/l}) - (H\cdot \pfrac{}{\rho} u_0) \circ \Psi_{1/l}}_{L^1(M)} \\
	&\leq&  C \norm{\pfrac{u_0}{\rho}(\eta_{1/l}(\rho),y)\eta'_{1/l}(\rho)}_{L^1(\set{\rho\leq 1/l})} + C \norm{\pfrac{u_0}{\rho}(\eta_{1/l}(\rho),y)}_{L^1(\set{\rho\leq 1/l})} \\
	&\leq& o(1).
\end{eqnarray*}
The essential part is
	\begin{eqnarray*}
		&& \norm{\frac{\partial^2}{\partial \rho^2} (u_0\circ \Psi_{1/l})- (\frac{\partial^2}{\partial \rho^2} u_0) \circ \Psi_{1/l}}_{L^1(\set{\rho\leq 1/l})} \\
		&\leq& C \norm{\pfrac{u_0}{\rho}(\eta_{1/l}(\rho),y) \eta''_{1/l}(\rho)}_{L^1(\set{\rho\leq 1/l})} + C\norm{(\frac{\partial^2 u_0}{\partial \rho^2}) (\eta_{1/l}(\rho),y)}_{L^1(\set{\rho\leq 1/l})}\\
		&\leq& C \int_{\partial M} \abs{\pfrac{u_0}{\nu}} dV_{\partial M} + o(1).
	\end{eqnarray*}

In summary, we have shown
\begin{equation*}
	\norm{\pfrac{u}{t}-v}_{L^1(M\times [0,T])}\leq C \int_{\partial M} \abs{\pfrac{u_0}{\nu}} dV_{\partial M}.
\end{equation*}

To conclude this section, we prove (\ref{eqn:end}). If we write $h_l= \partial_t u_l -v_l$, then we have
\begin{equation}\label{eqn:hl}
	\partial_t h_l = a \tilde{\triangle} h_l + \tilde{b} h_l + \bar{f}
\end{equation}
where
\begin{equation*}
	\tilde{b}= \partial_t a \cdot a^{-1} + b
\end{equation*}
and
\begin{equation*}
	\bar{f}= \left( \partial_t a\cdot a^{-1} b -\partial_t b \right)(u-u_l).
\end{equation*}
By the assumption, we know $\tilde{b}$ is bounded on $M\times [0,T]$ and by Theorem \ref{app:good}, we know $\norm{\bar{f}}_{C^0(M\times [0,T])}$ goes to zero as $l\to \infty$. 

We can write $h_l(x,0)=h_+-h_-$ for non-negative continuous functions $h_+$ and $h_-$. Similarly, $\bar{f}=f_+-f_-$. Let $h_{\pm}$ be the solution of (see Remark \ref{rem:continuous})
\begin{equation*}
	\partial_t h_{\pm} = a\tilde{\triangle} h_{\pm} + \tilde{b} h_{\pm} + f_{\pm}
\end{equation*}
with
\begin{equation*}
	\pfrac{h_\pm}{\nu}|_{\partial M}=0
\end{equation*}
and
\begin{equation*}
	h_\pm(x,0)= h_\pm.
\end{equation*}
As argued in Remark \ref{rem:maximum}, $h_\pm$ remains non-negative. By the uniqueness of continuous solution to (\ref{eqn:hl}), we have $h_l=h_+-h_-$.  To estimate $h_+$, we compute
\begin{eqnarray*}
	&& \frac{d}{dt} \int_{M} h_+ \cdot a^{-1} dV \\
	&\leq &  \int_{M} \tilde{\triangle} h_+ dV + \int_M h_+ \left( \partial_t (a^{-1}) + a^{-1}\tilde{b} \right) + a^{-1} \bar{f} dV \\
	&\leq& C \int_M h_+ dV + C \norm{\bar{f}}_{C^0(M\times [0,T])}.
\end{eqnarray*}

Integrating this inequality, we get for $t\in [0,T]$
\begin{equation*}
	\int_M h_+(t) dV \leq C(T) ( \norm{h_+(0)}_{L^1(M)} + \norm{\bar{f}}_{C^0(M\times [0,T])}) .
\end{equation*}
Similar results hold for $h_-$. This finishes the proof of (\ref{eqn:end}).

\end{document}